\documentclass[final,times,12pt]{elsarticle}
\usepackage[utf8]{inputenc}
\usepackage{amsmath}
\usepackage{amsthm}
\usepackage{amsfonts}
\usepackage{amssymb}
\usepackage{graphicx}
\usepackage{geometry}
\usepackage{float}
\theoremstyle{plain}

\newtheorem{defi}{Definition}[section]
\newtheorem{theo}{Theorem}[section]

\newtheorem{lem}{Lemma}[section]
\newtheorem{prop}{Proposition}[section]

\newcommand{\R}{\mathbb{R}}
\usepackage[figurename=Fig.]{caption}
\usepackage{makeidx}
\usepackage{geometry}
\usepackage{subcaption}
\journal{}

\begin{document}
	
	\begin{frontmatter}
		
		%% Title, authors and addresses
		
		%% use the tnoteref command within \title for footnotes;
		%% use the tnotetext command for theassociated footnote;
		%% use the fnref command within \author or \address for footnotes;
		%% use the fntext command for theassociated footnote;
		%% use the corref command within \author for corresponding author footnotes;
		%% use the cortext command for theassociated footnote;
		%% use the ead command for the email address,
		%% and the form \ead[url] for the home page:
		%% \title{Title\tnoteref{label1}}
		% \tnotetext[label1]{}
		%% \author{Name\corref{cor1}\fnref{label2}}
		%% \ead{email address}
		%% \ead[url]{home page}
		%% \fntext[label2]{}
		%% \cortext[cor1]{}
		%% \address{Address\fnref{label3}}
		%% \fntext[label3]{}
		
		\title{Approximation of the geodesic curvature and applications for spherical  geometric subdivision schemes}
		
		%% use optional labels to link authors explicitly to addresses:
		\author[1]{Aziz Ikemakhen. }
		\ead{ikemakhen@fstg-marrakech.ac.ma}
		
		\author[1]{Mohamed Bellaihou \corref{cor1}}
		\ead{mohamed.bellaihou@ced.uca.ma}
		
		\address[1]{Cadi-Ayyad University, FSTG BP 549 Marrakesh, Morocco.}
		%\address[2]{Cadi-Ayyad University, FSTG BP 549 Marrakesh, Morocco}
		\cortext[cor1]{Corresponding author}
		
		%\author{M. Bellaihou}
		
		%\address{}
		
		\begin{abstract}
			Many applications of geometry modeling and computer graphics necessite accurate curvature estimations of curves on the plane or on manifolds. In this paper, we define the notion of the discrete geodesic curvature of a geodesic polygon on a smooth surface. We show that, when a geodesic polygon P is closely inscribed on a $C^2$-regular curve, the discrete geodesic curvature of P estimates the  geodesic curvature of C.
			This result  allows us to evaluate the geodesic curvature of discrete curves on surfaces. In particular, we apply such result to planar and spherical 4-point angle-based subdivision schemes. We show that such schemes cannot generate in general $G^2$-continuous curves.
			We also give a novel example of $G^2$-continuous subdivision scheme on the unit sphere using only points and discrete geodesic curvature called curvature-based 6-point spherical scheme.
			
		\end{abstract}
		
		
	\end{frontmatter}
	\section{Introduction}
	In discrete differential geometry, the curvature of a plane or space curve at a point $p$ can be defined as the inverse of the radius of the oscillating circle. This one is the limit of the circle interpolating $p$ and two points moving closer to $p$ along the curve. The radii of the interpolating circle is often used as a discrete curvature of polygonal lines so as to estimate the pointwise curvature of plane curves \cite{Mo}.  Another discrete curvature associated to a polygon $P=\{p_0,..,p_n\}$ at a vertex $p_i$, involving the length of two consecutive edges and the exterior angle at $p_i$ (called the angular defect), were introduced in \cite{Bo} and  \cite{Bo2}. This notion were applied to
	shape blending or morphing of plane curves (\cite{Ma}, \cite{Su}).\\
For subdivision curves in the plane, instead of speaking about regular parametrized curves, we speak about $G^1$-continuous ones. Independently of any parametrization, one can refine iteratively a given polygon in order to approximate curves without cusps. Such schemes are said to be geometric. In \cite{Dy}, we find the so-called angle-based 4-point scheme generating planar $G^1$-limit curves using angle rules.\\
An example of geometric subdivision schemes in three-dimensional space is given in \cite{Ca}. Using a sequence of space points and vectors, one can define the new point and vector on the sphere determined by two consecutive points and their tangents. The scheme can produce $G^1$-limit curves and when the data are sampled from a common sphere the scheme provides a spherical curve.\\
The first $G^2$-continuous planar geometric subdivision scheme was introduced in \cite{De}. For spatial $G^2$-curves we find in \cite{We} and \cite{Ca} examples of those schemes. Unfortunately, the evidence for such schemes to be $G^2$-continuous is only numerical. Theoretical proof remains very difficult. For the plane case, E. Volontè gave in her thesis \cite{Vo} sufficient condition for a planar geometric subdividion scheme to be $G^2$-continuous. But the proof is incomplete. \\
In order to design spherical curves refining only data points on a sphere, the authors in \cite{Be} gave a purely geometrical definition of a spherical interpolatory subdivision scheme. They propose a spherical generalization of the planar incenter subdivision scheme \cite{De} refining point-vector data by proving the convergence and $G^1$-continuity. \\
In this paper, we  define the notion of discrete geodesic curvature for a geodesic polygon on a smooth surface (see definition 2.2). We prove the main result that stipulates for a geodesic polygon
$P=\{p_{0}, \ldots, p_{n}\}$ closely inscribed on a
$C^2$-regular curve $C$, the discrete geodesic curvature  of $P$   estimates the  geodesic curvature of $C$ at the sample points $p_i$.
We give then the spherical generalization of the angle-based 4-point scheme generating spherical curves without involving Hermite data as opposite to \cite{De} and \cite{Be}. We prove that the subdivision scheme is convergent and $G^1$-continuous in the sense of \cite{Be}. We show also that the spherical angle-based 4-point scheme cannot be in general $G^2$-continuous using the discrete geodesic curvature estimation.\\
The rest of the paper is organized as follows: Section 2 is devoted to define the discrete geodesic curvature of an embedded polygon on a surface and to prove that it is an approximation of the geodesic curvature of a regular curve.
In section 3, we give a spherical generalization of the geometric 4-point scheme \cite{Dy} and we prove that it is convergent and $G^1$-continuous. We also investigate the discrete geodesic curvature to show that the proposed spherical scheme cannot be in general $G^2$-continuous. In the end, we give a novel example of $G^2$-continuous scheme using only points and discrete geodesic curvatures called the curvature-based 6-point spherical scheme.
	\section{Discrete geodesic curvature of embedded curves on surfaces}
	In this section, we will prove the result in Theorem \ref{a3}. For that, we need some definitions.\\
	Let  $ \sigma : I \rightarrow S $ be a regular $C^2$-curve parametrized by arc length in some oriented smooth surface $S$ with Gauss map $ \mathbf{n} : S \rightarrow \mathbb{S}^{2} $. The geodesic curvature at $p=\sigma(t)$ is defined  by:
	\begin{equation}
		\kappa_{g}(p):=<\ddot{\sigma}(t),\mathbf{n}(p)\wedge \dot{\sigma}(t)>=<\Big(\frac{D\dot{\sigma}}{dt}\Big)(t),\mathbf{n}(p)\wedge \dot{\sigma}(t)>,
		\label{GeodCurv}
	\end{equation}
	where $\displaystyle \frac{D}{dt}$ is the covariant derivative on $S$, and we have:
	$$
	\ddot{\sigma}(t)=\Big(\frac{D\dot{\sigma}}{dt}\Big)(t)+<\ddot{\sigma}(t)\wedge \mathbf{n}(p)>\mathbf{n}(p).
	$$
	
	\begin{defi}\label{d0}
		A geodesic V-line  $\mathbf{V}:=\{p_1,p,p_2\} $   on a surface $S$  is made of two geodesics $c_1$ and $c_2$  connecting  respectively $p_1$ to $p$ and $p$ to  $p_2$ (see Fig. \ref{fig:p1}).\\
		If the surface is plane, the geodesics are rectilinear segments and $\mathbf{V} $  is simply called a V-line.
	\end{defi}
	
	\begin{defi}\label{d1}
		\begin{enumerate}
			\item Let  $\mathbf{P}:= \{q_1,q,q_2\} \subset \R^2$ be a   V-line.  The  discrete curvature of $\mathbf{P}$ at $q$ is defined by
			$$
			\displaystyle \kappa_{d}(q_1,q,q_2):=\frac{2\, \tilde{\delta}}{qq_1+qq_2} ,$$
			where $\tilde{\delta}$ denotes the angle  $\sphericalangle(\overrightarrow{q_1q},\overrightarrow{qq_2})$, called the angular defect.
			\item Let  $\mathbf{L}= \{p_1,p,p_2\} \subset S$ be a  geodesic V-line. We define the discrete geodesic curvature of $\mathbf{L} $  at $p$  by:
			$$\kappa_g(p_1,p,p_2):=\frac{2\, \delta}{l_1+l_2},$$
			where $\delta:=\sphericalangle(u_1,u_2)$ is the geodesic angular defect, and $u_i$ the tangent vector at $p$ of the geodesic $c_i$. $ l_i$ is the  length of $c_i$.
		\end{enumerate}
		
	\end{defi}
	
	\begin{theo}\label{a3}
		Let $\sigma : I \rightarrow S$ be a  regular $C^2$-curve on a smooth surface $S$ and  $p$ be a point on $\sigma(I)$.  Then its geodesic  curvature $\kappa(p) $  at $p$ can be obtained as a limit of discrete geodesic curvature of geodesic V-line  $\mathbf{P}=\{p_1,p,p_2\} $ inscribed on it, i.e:
		$$ \displaystyle \lim_{\substack{ p_1,p_2 \to p \\ p_1, p_2 \in \sigma(I)} }\, \kappa_g(p_1,p,p_2)=\kappa(p).$$

	\end{theo}
	In \cite{Bo}, the authors showed this theorem in the planar case for a  $C^3$- curve, using an analytic proof. In the following,  we begin to show  this case, assuming only that the curve is of class $C^2$, using a very simple geometric proof.
	\begin{lem}\label{a0}(\cite{Bo})\\
		Let $\gamma  : I \rightarrow \mathbb{R}^2$ be a planar regular $C^2$-curve and  $p$ be a point on it.
		Then its curvature $\kappa(p)$  at $p$ can be obtained as a limit of   discrete curvature  of V-line $\{p_1,p,p_2\} $ inscribed on it:
		$$\lim_{\substack{p_1,p_2 \to p, \\ p_1,p_2 \in \gamma(I)} }\, \kappa_{d}(p_1,p,p_2)=\kappa(p).$$
	\end{lem}
	\begin{proof}
		We know that the curvature of the circumscribed circle defined by the V-line $ \{p_1,p,p_2\}$ is given by (see Fig. \ref{fig:p2}):
		\begin{equation*}
			\rho_p(p_1,p,p_2)=\displaystyle \frac{2\, sin\, \alpha_p}{pp_2}=\frac{2\, sin\, \beta_p}{pp_1}.
		\end{equation*}
		Then
		$$
		\rho_p(p_1,p,p_2) =\frac{2( sin \, \alpha_p +sin \, \beta_p)}{pp_2+pp_1}.
		$$
		Hence
		$$\rho_p(p_1,p,p_2)=\displaystyle \frac{4\, sin\bigg(\frac{\alpha_p+\beta_p}{2}\bigg)\,  cos\bigg(\frac{\alpha_p-\beta_p}{2}\bigg) }{pp_2+pp_1}=\frac{4\, sin\bigg(\frac{\delta_p}{2}\bigg)\,  cos\bigg(\frac{\alpha_p-\beta_p}{2}\bigg) }{pp_2+pp_1},$$
		where $\delta_p=\alpha_p+\beta_p$. So
		\begin{equation}
			\rho_p(p_1,p,p_2)=\displaystyle \kappa_d(p_1,p,p_2) \; cos(\alpha_p-\beta_p) \displaystyle \frac{sin(\delta_p/2)}{\delta_p/2}.
			\label{44}
		\end{equation}
		Since $\alpha_p, \, \beta_p \to 0$ as $p_1, \, p_2 \to p$, we conclude that
		$$ \displaystyle \lim_{p_1,p_2 \to p}\, \kappa_d (p_1,p,p_2)=\kappa(p).$$
	\end{proof}
	\begin{lem}\label{a1}
		Let $\sigma : I \rightarrow S$ be a $C^2$-regular curve, parametrized by its arc length,  lies on a smooth surface $S$ and $p=\sigma(0)$ be a point on  the curve.
		Let $\gamma$ be the image curve of $\sigma$ by the inverse of  the exponential map at $p$, given by $\gamma(s)=exp_{p}^{-1}(\sigma(s))$ and $p=\gamma(0)$.  Then we have:
		\begin{enumerate}
			\item $$\dot{\gamma}(0)=\dot{\sigma}(0), \qquad \ddot{\gamma}(0)=\Big(\frac{D\dot{\sigma}}{dt}\Big)(0).$$
			\item The geodesic curvature of $\sigma $ at $p$ is equal to the curvature of $\gamma$ at the origin:
			$$
			\kappa_{g,\sigma}(p)=\kappa_{ \gamma}(p).
			$$
		\end{enumerate}
	\end{lem}
	\begin{figure}
		
		\begin{minipage}[c]{0.45\linewidth}
			\includegraphics[width=1.2\linewidth]{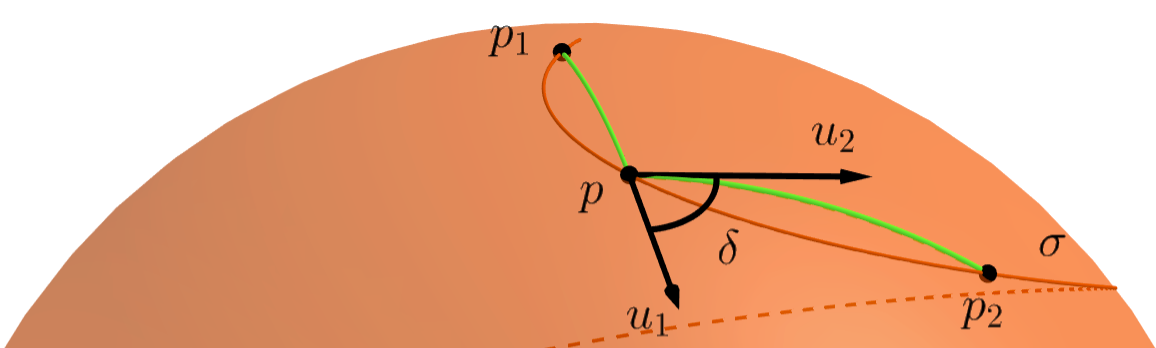}
			\caption{A geodesic V-line inscribed on $\sigma$.}
			\label{fig:p1}
		\end{minipage}
		\hfill
		\hspace*{1.5cm}
		\begin{minipage}[c]{0.5\linewidth}
			\includegraphics[width=0.9\linewidth]{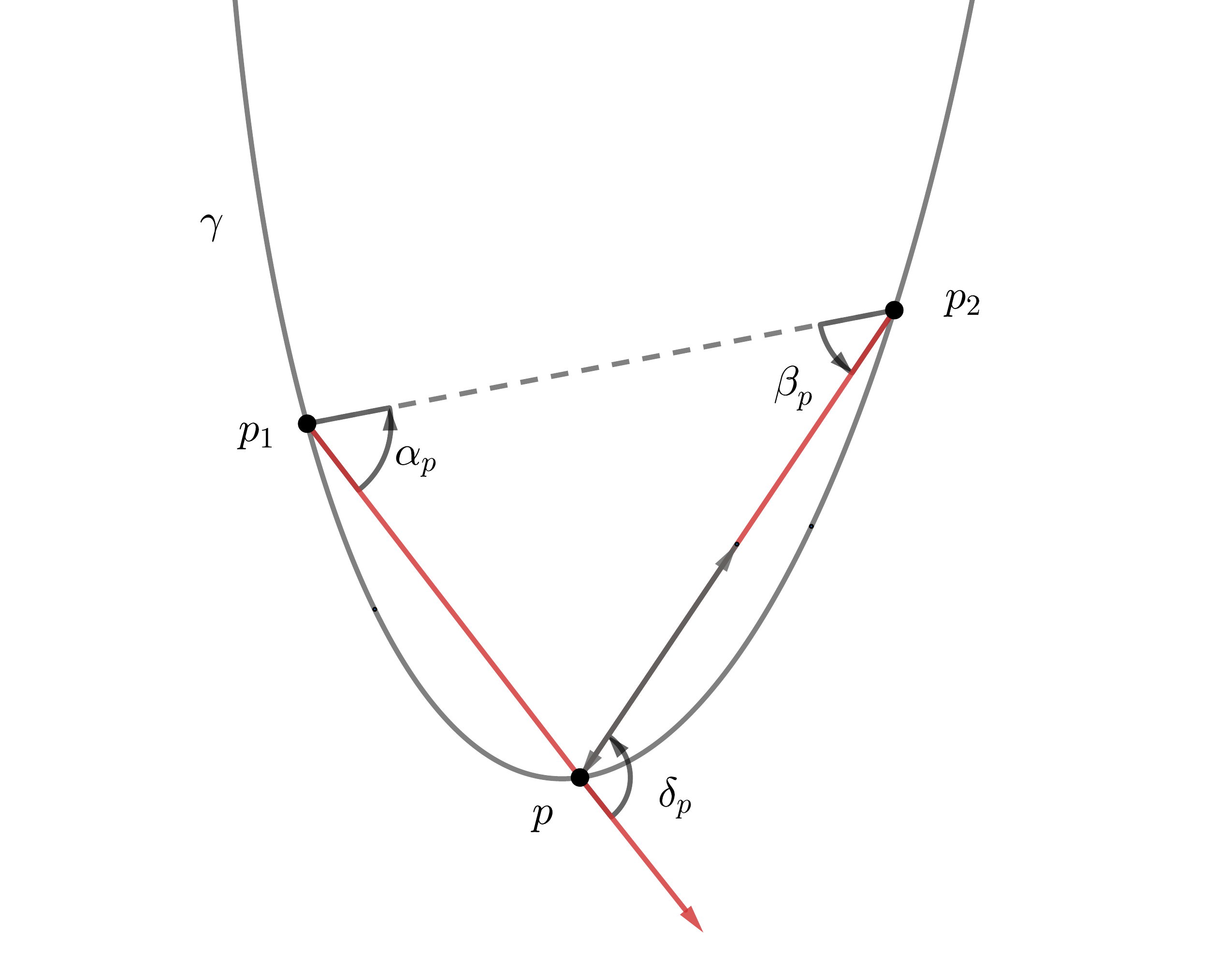}
			\caption{Proof of Lemma \ref{a0}.}
			\label{fig:p2}
		\end{minipage}
	\end{figure}
	\begin{proof}
		\begin{enumerate}
			\item Let $(x^1,x^2)$ be the normal coordinates defined on a neighbourhood of $p$  associated to an orthonormal basis $(e_1,e_2)$ of $T_{p} S$ . Then $exp^{-1}_p(x) = x^1 e_1 +  x^2 e_2$ and $\gamma(s) = \sigma^1(s) e_1 +  \sigma^2(s) e_2$, where
			$\sigma^i= x^i \circ \sigma$. From the fact that  Christoffel's symbols $\Gamma_{i,j}^k(p)=0$, we get
			$$
			\dot{\gamma}(t) = \dot{\sigma^1}(t)  e_1 +  \dot{\sigma^2}(t)   e_2= \dot{\sigma}(t) ,
			$$
			and
			$$
			\Big(\frac{D\dot{\sigma}}{dt}\Big)(0)= \displaystyle \sum_{i,j,k}  \bigg( \ddot{\gamma}^k(0) +\dot{\gamma}^i(0) \dot{\gamma}^j(0) \Gamma_{i,j}^k(p) \bigg) e_k = \ddot{\gamma}(0).
			$$
			\item From \eqref{GeodCurv}, we have
			$$\kappa_{g,\sigma}(p):=<\Big(\frac{D\dot{\sigma}}{dt}\Big)(0),\mathbf{n}(p)\wedge \dot{\sigma}(0)>
			=<\ddot{\gamma}(0),\mathbf{n}(p)\wedge \dot{\gamma}(0)>=\kappa_{ \gamma}(p).
			$$
		\end{enumerate}
	\end{proof}
	
	\begin{proof}[Proof of Theorem \ref{a3}]
		Since radial geodesics are mapped isometrically to straight lines in $T_pS$ by $exp^{-1}_p$ (and so is the angle between them), we have $\kappa_g(p_1,p,p_2)=\displaystyle \frac{2\, \delta}{pq_1+pq_2}$ with $q_i=exp_p^{-1}(p_i)$. Then,  by Lemma \ref{a0} and Lemma \ref{a1}, we conclude the claim.	
	\end{proof}
	\section{Spherical interpolatory geometric subdivision schemes}
	Spherical interpolatory geometric subdivision schemes on the unit sphere \cite{Be} are a generalization of the planar ones \cite{Dy}. Using the exponential map of the sphere and ASA-formula of triangles, one can define geometrically new points on $\mathbb{S}^2$.
	In this section, we recall materials found in \cite{Be}.
	In Section 3.1, we define the angle-based 4-point scheme on the unit sphere. We prove that the proposed scheme is convergent and $G^1$-continuous. We show that the scheme cannot be in general $G^2$-continuous using the discrete geodesic curvature. \\
	The second part of this section will concern a second novel spherical scheme capable of producing $G^2$-continuous curves. The so-called curvature-based 6-point spherical subdivision scheme employs the discrete geodesic curvature to insert new points. \\
	Let $\mathbb{S}^2$ be the oriented sphere and $d$ its spherical distance.
	On the geodesic polygon $P^j:=\{p_i^j / \, \, p_i^j \in \mathbb{S}^2 \}$ at $j^{st}$ iteration, let $T_i^j=\Delta p_{i-1}^jp_i^jp_{i+1}^j $ be the spherical triangle at $i^{st}$ position. We associate (see Fig. \ref{fig:f10}):
	\begin{itemize}
		\item The geodesic $c_{i,k}^j$ of arc length $l_{i,k}^j$ connecting $c_{i,k}^j(0)=p_i^j$ and $c_{i,k}^j(l_{i,k}^j)=p_k^j$.
		\item The edge length:
		\begin{equation}
			e_i^j=d(p_i^j,p_{i+1}^j).
			\label{11}
		\end{equation}
		\item The unit tangent vectors to the geodesics $c_{i,i+1}^j$ and $c_{i-1,i}^j$ at $p_i^j$:
		\begin{equation}
			U_i^j=\dot{c}_{i,i+1}^j(0) \qquad \text{and} \qquad V_i^j=\dot{c}_{i-1,i}^j(l_{i-1,i}^j).
			\label{12}
		\end{equation}
		\item The signed angles:
		\begin{equation}
			\alpha_i^j=\sphericalangle (\dot{c}_{i-1,i+1}^j(0),U_{i-1}^j) \qquad \text{and} \qquad \beta_i^j=\sphericalangle(V_{i+1}^j,\dot{c}_{i-1,i+1}^j(l_{i,i+1}^j)),
			\label{13}
		\end{equation}
		with the signed angular defect:
		\begin{equation}
			\delta_i^j=\sphericalangle (U_i^j,V_i^j).
			\label{14}
		\end{equation}
		\item The signed area $A_i^j$ of $T_i^j$ given by Girard's formula:
		\begin{equation}
			A_i^j=\alpha_i^j+\beta_i^j-\delta_i^j.
			\label{42}
		\end{equation}
	\end{itemize}

			%%%%%%%%%%%%%%%%%%%%%%%%%%%%%%%%%%%%%%%%%%%%%%%%%%%%%%%%%%%%%%%%%%%%%%%%%%%%
	\begin{figure}[H]
		\begin{minipage}[c]{0.5\linewidth}
			\includegraphics[width=1\linewidth]{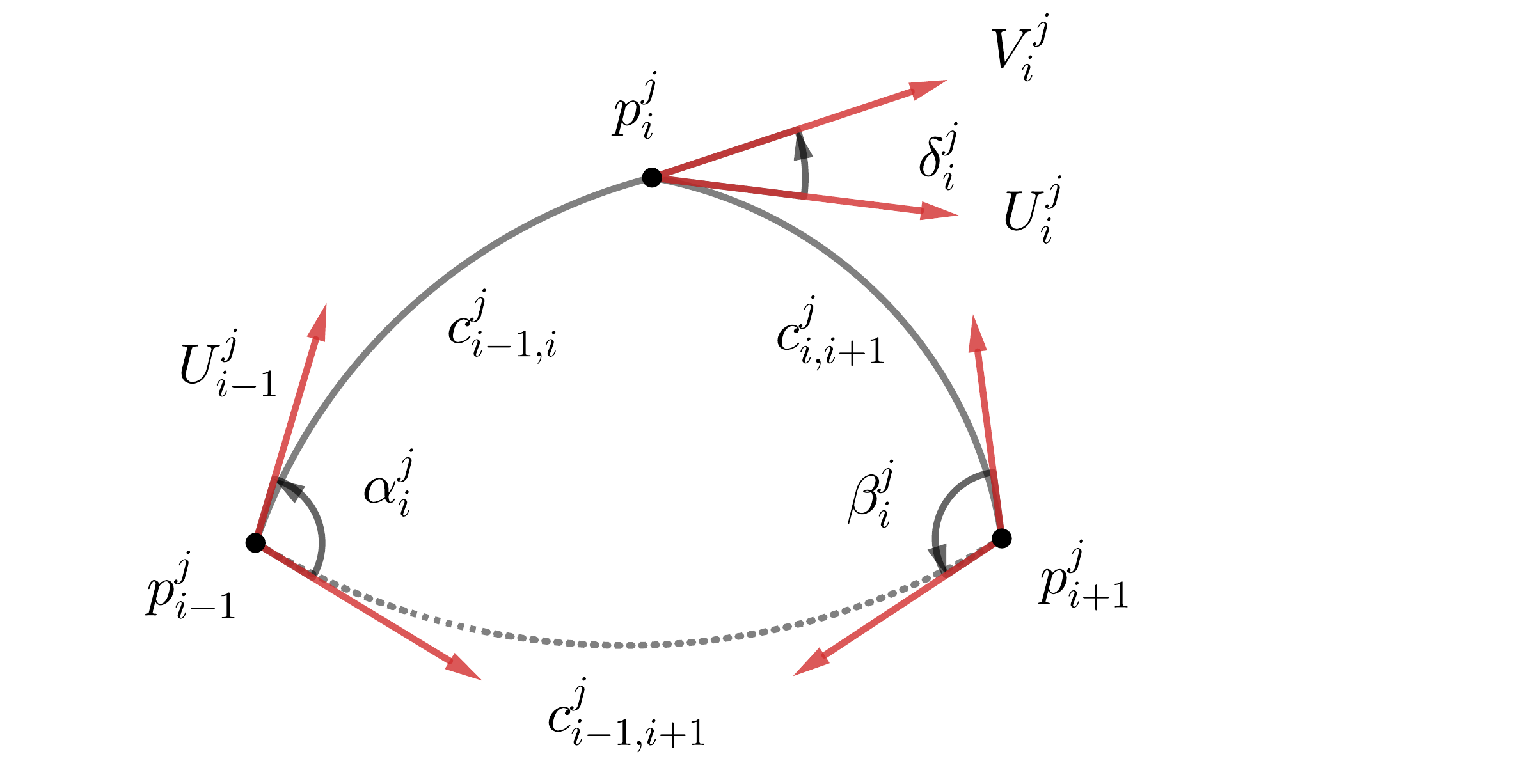}
			\caption{The spherical triangle $T_i^j$.}
			\label{fig:f10}
		\end{minipage}
		\hfill
		\hspace*{-2cm}
		\begin{minipage}[c]{0.5\linewidth}
			\includegraphics[width=0.9\linewidth]{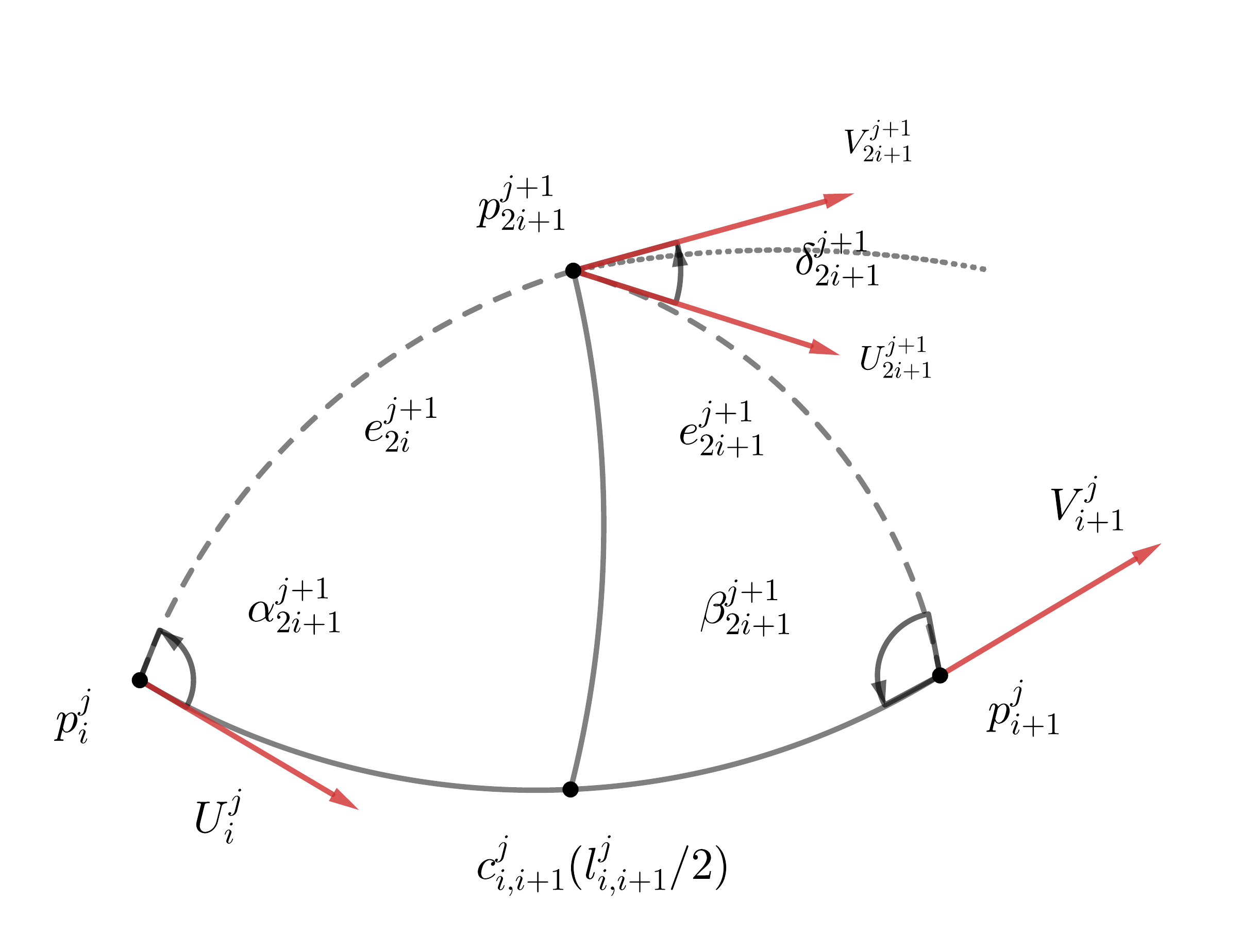}
			\caption{The new point $p_{2i+1}^{j+1}$.}
			\label{fig:f11}
		\end{minipage}
	\end{figure}
To define a spherical interpolatory geometric subdivision scheme, we need to express explicitly the new point $p_{2i+1}^{j+1}$ from the data $\alpha_{2i+1}^{j+1}$, $\beta_{2i+1}^{j+1}$ and $l^j_{i,i+1}$ in the triangle $p_{i}^j\, p_{2i+1}^{j+1}\, p_{i+1}^j$ (see Fig. \ref{fig:f11}). \\
First, the plane $\Pi$ containing $\overrightarrow{op_i^j}$ and $\overrightarrow{op_{2i+1}^{j+1}}$ ($o$ is the centre of the sphere) is given by $\overrightarrow{op^j_{i}}$ and $ R_{p^j_{i},\alpha_{2i+1}^{j+1}}  U^j_{i}$ the rotation of $U^j_{i}$ around $\overrightarrow{op^j_{i}}$ by angle $\alpha^{j+1}_{2i+1}$. \\
To get $p_{2i+1}^{j+1}$, it remains to rotate $\overrightarrow{op^j_{i}}$ on the plane $\Pi$ by angle $e_{2i}^{j+1}$. Then:\\
\begin{equation}
	p_{2i+1}^{j+1}=R_{\vec{n},e_{2i}^{j+1}} \Big(  \overrightarrow{op^j_{i}}  \Big),
	\label{15}
\end{equation}
where  $\vec{n}=\overrightarrow{op^j_{i}} \wedge  R_{p^j_{i},\alpha_{2i+1}^{j+1}}  U^j_{i}$ the normal vector of the plane $\Pi$. \\
In a symmetric way, we have:
\begin{equation}
	p_{2i+1}^{j+1}=R_{\vec{m},e_{2i+1}^{j+1}} \Big(  \overrightarrow{op^j_{i+1}}  \Big),
	\label{15a}
\end{equation}
where  $\vec{m}=\overrightarrow{op^j_{i+1}} \wedge  R_{p^j_{i+1},\pi-\beta_{2i+1}^{j+1}}  V^j_{i+1}$. \\
By taking
\begin{equation}
	p_{2i}^{j+1}=p_i^j,
	\label{odd}
\end{equation}
we have the following definitions:
	\begin{defi}[\cite{Be}]\label{DefSIGS}
		We define a Spherical Interpolatory Geometric Subdivision (SIGS) scheme by \eqref{15} and \eqref{odd} (or equivalently by \eqref{15a} and \eqref{odd}), where $e^{j+1}_{2i}$ and $e_{2i+1}^{j+1}$ are given by ASA-formulas of spherical triangles:
				\begin{equation*}
			\left\{
			\begin{aligned}
				& e_{2i}^{j+1}=   \displaystyle Arctan \Bigg( \displaystyle \frac{2sin(\beta_{2i+1}^{j+1})}{cot( \frac{l^j_{i,i+1}}{2}) sin(\alpha_{2i+1}^{j+1}+\beta_{2i+1}^{j+1})+tan( \frac{l^j_{i,i+1}}{2}) sin(\alpha_{2i+1}^{j+1}-\beta_{2i+1}^{j+1})}\Bigg)            ,\\
				&e_{2i+1}^{j+1}=\displaystyle Arctan \Bigg( \displaystyle \frac{2sin(\alpha_{2i+1}^{j+1})}{cot( \frac{l^j_{i,i+1}}{2}) sin(\alpha_{2i+1}^{j+1}+\beta_{2i+1}^{j+1})+tan( \frac{l^j_{i,i+1}}{2}) sin(\beta_{2i+1}^{j+1}-\alpha_{2i+1}^{j+1})}\Bigg),  \\
			\end{aligned}
			\right.
			\label{}
		\end{equation*}
		 and $\alpha_{2i+1}^{j+1}$ and $\beta_{2i+1}^{j+1}$ are both non-zero, $R_{\alpha_{2i+1}^{j+1},p_i^j}$ the rotation around $\overrightarrow{op_i^j}$ by angle $\alpha_{2i+1}^{j+1}$.\\
		For the special case when $\alpha_{2i+1}^{j+1}=\beta_{2i+1}^{j+1}$ we said that the SIGS scheme is bisector (SBIGS-scheme), and is defined by \eqref{15} and \eqref{odd} (or \eqref{15a} and \eqref{odd}), where:
		$$e_{2i}^{j+1}=e_{2i+1}^{j+1}=\displaystyle Arctan \bigg(\frac{tan(\displaystyle \frac{l^j_{i,i+1}}{2})}{cos(\alpha_{2i+1}^{j+1})}\bigg)$$.
	\end{defi}
Note that Definition \ref{DefSIGS} is similar to the one in \cite{Be}. The only difference stands in Eqs. \eqref{15} and \eqref{15a} in which in \cite{Be} they were expressed in terms of the exponential map of the unit sphere, while in this paper, for simplicity, we propose to write them using rotations.\\
	Now, let $$\delta^j:=\underset{i\in \mathbb{Z}}{sup}\, |\delta_{i}^{j}|.$$
	We recall Corollary 5.1 of \cite{Be} that provides sifficient condition for the convergence and $G^1$-continuity of SBIGS-schemes:
	\begin{theo}[\cite{Be}]\label{co1}
		If the sequence $\{\delta^j\}_j$ is summable, then the SBIGS-scheme is convergent and the limit curve is $G^1$-continuous.
	\end{theo}
	\subsection{The spherical angle-based 4-point scheme}
	The spherical angle-based 4-point scheme is a SBIGS-scheme in which the new angles are given in terms of two consecutive angular defects $\delta_i^j$ and $\delta_{i+1}^j$ of the four points $p_{i-1}^j,\, p_{i}^j, \, p_{i+1}^j$ and $p_{i+2}^j$. Namely:
	\begin{equation}
		\alpha_{2i+1}^{j+1}=\beta_{2i+1}^{j+1}=\displaystyle \frac{\delta_i^j+\delta_{i+1}^j}{8}.
		\label{41}
	\end{equation}
	From \eqref{42}, we have
	\begin{equation}
		A_{2i+1}^{j+1}=2\alpha_{2i+1}^{j+1}-\delta_{2i+1}^{j+1}.
		\label{46}
	\end{equation}
	From \eqref{41}, we get $$\delta_{2i+1}^{j+1}=\displaystyle\frac{1}{4}(\delta_i^j+\delta_{i+1}^j)-A_{2i+1}^{j+1}.$$
	As we see in Fig. \ref{fig:p5}, we have
	$$\delta_{2i}^{j+1}=\delta_i^j-\alpha_{2i+1}^{j+1}-\alpha_{2i-1}^{j+1}.$$
	Then
	$$ \delta_{2i}^{j+1}=\displaystyle\frac{1}{8}(-\delta_{i-1}^j+6\delta_i^j-\delta_{i+1}^j).$$
	Consequently
	\begin{equation}
		\left\{
		\begin{aligned}
			&\delta_{2i+1}^{j+1}=\displaystyle\frac{1}{4}(\delta_i^j+\delta_{i+1}^j)-A_{2i+1}^{j+1},           \\
			&\delta_{2i}^{j+1}=\displaystyle\frac{1}{8}(-\delta_{i-1}^j+6\delta_i^j-\delta_{i+1}^j).   \\
		\end{aligned}
		\right.
		\label{36}
	\end{equation}
	\begin{figure}
		\centering
		\includegraphics[width=0.7\linewidth]{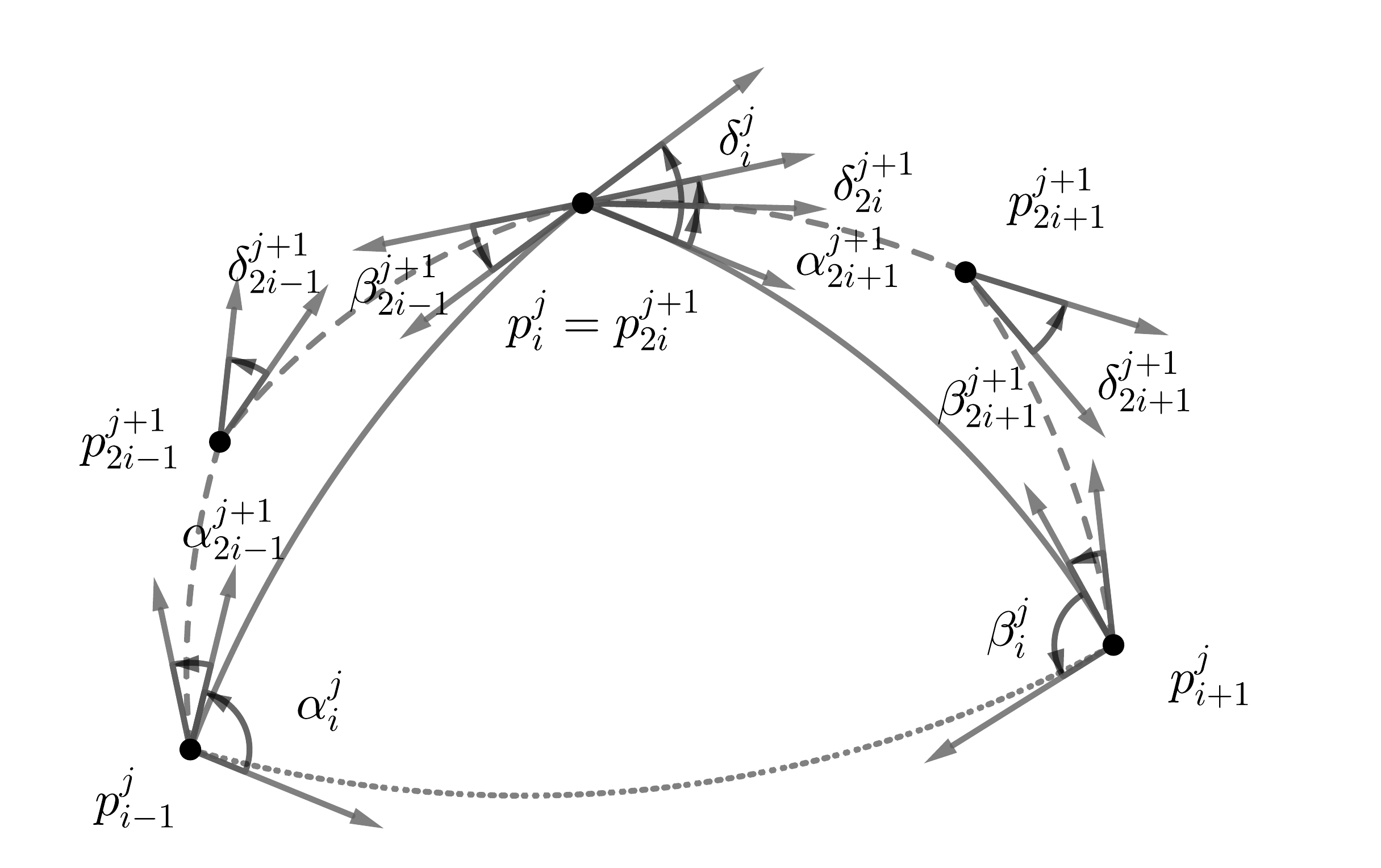}
		\caption{A first iteration on the triangle $T_i^j$.}
		\label{fig:p5}
		
	\end{figure}
	
	In virtue of Theorem \ref{co1}, the summability of the sequence $\{  \delta^j\}_j$ is sufficient for the convergence and $G^1$-continuity of the 4-point SBIGS-scheme. To prove that $\{\delta^j \}_j$ is summable, it is sufficient to show that for some $\nu<1$ and $p\geqslant 0$, we have $\delta^{j+p}\leqslant \nu \, \delta^j$.
	\begin{theo}\label{a5}
		The 4-point SBIGS-scheme is convergent and $G^1$-continuous.
	\end{theo}	
	\begin{proof}
		Since all angles in \eqref{46} have the same sign, we can write:
		$$|\delta_{2i+1}^{j+1}|=2|\alpha_{2i+1}^{j+1}|-|A_{2i+1}^{j+1}|.$$
		Then from \eqref{41} we get
		\begin{equation}
			|\delta_{2i+1}^{j+1}| \leqslant 2|\alpha_{2i+1}^{j+1}| \leqslant \displaystyle
			\frac{1}{4} \, (|\delta_i^j|+|\delta_{i+1}^j|).
			\label{33}
		\end{equation}
		Hence
		\begin{equation}
			|\delta_{2i+1}^{j+1}| \leqslant \displaystyle \frac{1}{2} \, \delta^j.
			\label{27}
		\end{equation}
		On the other hand, by \eqref{36}, we have
		\begin{equation}
			|\delta_{2i}^{j+1}|  \leqslant \displaystyle  \frac{1}{8}|\delta_{i-1}^j|+ \frac{6}{8}|\delta_i^j|+\frac{1}{8}|\delta_{i+1}^j| \leqslant \delta^j.
			\label{29}
		\end{equation}
		For the $j+2^{st}$ step, we have by \eqref{33}:
		\begin{equation*}
			|\delta_{4i+1}^{j+2}| \leqslant \displaystyle \frac{1}{4} \Big( |\delta_{2i}^{j+1}|+|\delta_{2i+1}^{j+1}| \Big), \qquad
			|\delta_{4i+3}^{j+2}| \leqslant \displaystyle \frac{1}{4} \Big( |\delta_{2i+1}^{j+1}|+|\delta_{2i+2}^{j+1}| \Big).
		\end{equation*}
		Using \eqref{27} and \eqref{29} we get
		\begin{equation}
			|\delta_{4i+1}^{j+2}| \leqslant \displaystyle \frac{3}{8} \delta^j, \qquad
			|\delta_{4i+3}^{j+2}| \leqslant \displaystyle \frac{3}{8} \delta^j.
			\label{31}
		\end{equation}
		Likewise
		\begin{equation*}
			\delta_{4i}^{j+2}=\displaystyle \frac{6}{8} \delta_{2i}^{j+1}-\frac{1}{8} \delta_{2i-1}^{j+1}-\frac{1}{8} \delta_{2i+1}^{j+1}, \qquad
			\delta_{4i+2}^{j+2}=\displaystyle \frac{6}{8} \delta_{2i+1}^{j+1}-\frac{1}{8} \delta_{2i}^{j+1}-\frac{1}{8} \delta_{2i+2}^{j+1} .
		\end{equation*}
		Then
		\begin{equation}
			|\delta_{4i}^{j+2}| \leqslant \displaystyle \frac{7}{8} \delta^j , \qquad
			|\delta_{4i+2}^{j+2}| \leqslant \displaystyle \frac{5}{8} \delta^j.
			\label{32}
		\end{equation}
		Finally, from \eqref{31} and \eqref{32} we conclude that
		$$\delta^{j+2} \leqslant \displaystyle \frac{7}{8} \, \delta^j.$$
		And the summability of $\{\delta^j \}_j$ holds.
	\end{proof}
	Now, recall the following theorem about small spherical triangles:
	\begin{theo}[ \cite{Le}, \cite{Na}]\label{c12}
		Let $T$ be a spherical triangle with small geodesic lengths $l_i$ , $i\in \{1,2,3\}$ and $\tilde{T}$ be the Euclidean triangle with the same lengths as those of $T$. If $\alpha_i$, $i\in \{1,2,3\}$, denote the geodesic angles of $T$, and $\tilde{\alpha_i}$ the corresponding angles of $\tilde{T}$. Then for $i \in \{1,2,3\}$ we have:
		\begin{equation}
			\tilde{\alpha_i}=\alpha_i-\displaystyle \frac{A(T)}{3}+o(l^4)
			\label{26},
		\end{equation}
		where $A(T)$ is the area of $T$ and $l=\underset{i}{sup} \, l_i$.
	\end{theo}
To define the $G^2$-continuity of spherical curves, we need to assume that the curve is $G^1$-continuous in the sense of \cite{Be}.
\begin{defi}
	Let $C$ be a $G^1$-continuous curve on the sphere and let $p$, $p_1$ and $p_2$ three points on the curve. $C$ is said to be $G^2$-continuous at $p$ if the limit $\displaystyle \lim_{\substack{ p_1,p_2 \to p \\ p_1, p_2 \in C} }\, \kappa_g(p_1,p,p_2):=\kappa_g(p)$ exists, when $p_1$ approaches p along the curve from the left and $p_2$ approaches p along the curve from the right. \\
	$C$ is said to be $G^2$-continuous if the map $p \in C \to \kappa_g(p)$ is well defined and continuous.
\end{defi}
	Let \begin{equation}
		\kappa_i^j:= \kappa_g(p_{i-1}^j,p_i^j,p_{i+1}^j)=\displaystyle \frac{2 \delta_i^j}{e_{i-1}^j+e_i^j}.
		\label{45}
	\end{equation}
	\begin{theo}\label{a6}
		The 4-point SBIGS-scheme is not in general $G^2$-continuous.
	\end{theo}
	\begin{proof}
		We need to prove that there exists some initial geodesic polygon $P^0$ such that the limit curve \\ $C:=\displaystyle \lim_{j \to +\infty} \, P^j$ of the 4-point SBIGS-scheme has unbounded discrete geodesic curvature.\\	
		Let $P^0=\{p_{-2},p_{-1},p,p_1,p_2\}$ be a geodesic polygon with $\delta_{-1}^0=\delta_1^0$ and $\delta_{0}^0 \neq \delta_{1}^0$ (see Fig. \ref{fig:p4}). Let $\kappa_0^j$ be the discrete geodesic curvature at $p$ of $P^j$. We will prove that: If $\delta_0^0>\delta_1^0$ (resp.$\delta_0^0<\delta_1^0$ ) then $\lim_{j \to \infty}\, \kappa_0^{j}=+\infty$ (resp. $-\infty$).\\
		According to Theorem \ref{a5}, we have $\displaystyle \lim_{j \to \infty} \delta^j=0$. By Corollary 4.1 and Proposition 4.4 of \cite{Be}, we have $\displaystyle \lim_{j \to \infty} e^j=\lim_{j \to \infty}\, A^j=0$ with $e^j:=\underset{i\in \mathbb{Z}}{sup}\, |e_{i}^{j}|$ and $A^j:=\underset{i\in \mathbb{Z}}{sup}\, |A_{i}^{j}|$.\\
		Theorem \ref{c12} tells us that, when spherical polygons $P^j$ have sufficiently small edge lengths, then we can consider them to be planar ones with plane angular defects (denoted by a tilde) are given by \eqref{26}:
		\begin{equation}
			\tilde{\delta}_0^j=\delta_0^j+\displaystyle \frac{A_0^j}{3}+o((e^j)^4),
			\label{23}
		\end{equation}
		for large $j$. Because $\delta_{-1}^0=\delta_{1}^0$, we can easily show that for all $j$ we have $e_{-1}^j=e_0^j$. Then
		$$ \kappa_0^{j}= \displaystyle \Big( \frac{\tilde{\delta}_0^{j}}{e_0^{j}}-\frac{A_0^j}{3\, e_0^j}\Big)+o((e^j)^3).$$
		Since in a neighbourhood of $p$, the 4-point SBIGS-scheme behaves like a planar one, by Proposition \ref{t1}, we have $\displaystyle \lim_{j \to \infty}\, \frac{\tilde{\delta}_0^{j}}{e_0^{j}}=\pm \infty$.\\
		On the other hand, the spherical area $A_0^{j}$ of the triangle $T_0^{j}$ is given by
		\begin{equation}
			\displaystyle sin \Big(\displaystyle\frac{A_0^{j}}{2}\Big)=\frac{\displaystyle sin \Big(\frac{e_0^j}{2} \Big)^2 \, sin(\delta_0^j)}{4\, \displaystyle cos \Big(\frac{d(p_{-1}^{j},p_1^j)}{2} \Big)}.
			\label{38}
		\end{equation}
		Then
		$$\displaystyle \lim_{j \to \infty}\, \frac{\displaystyle  sin \Big(\frac{A_0^{j}}{2}\Big)}{ \displaystyle sin \Big(\frac{e_0^j}{2}\Big)}=\lim_{j \to \infty} \, \frac{\displaystyle sin \Big(\displaystyle \frac{e_0^j}{2}\Big) \, sin(\delta_0^j)}{4\, \displaystyle cos \Big(\frac{d(p_{-1}^{j},p_1^j)}{2} \Big)}=0.$$
		Now, we can suppose that $\displaystyle |A_0^{j}|, \, \, |e_0^{j}|<\frac{\pi}{2}$.
		Then $\displaystyle \lim_{j \to \infty}\, \frac{\displaystyle  A_0^{j}}{ \displaystyle e_0^j}=0$.
		Hence, we deduce the claim:
		$$\lim_{j \to \infty}\, \kappa_0^{j}=\lim_{j \to \infty}\, \displaystyle \Big( \frac{\tilde{\delta}_0^{j}}{e_0^{j}}-\frac{A_0^j}{3\, e_0^j}\Big)=\pm \infty.$$
	\end{proof}

	To prove Theorem \ref{a6} numerically, we plot the discrete geodesic curvature values (red) of $P^j$ in both cases: Fig. \ref{fig:sph2} with $\delta_0^0>\delta_1^0$ and Fig. \ref{fig:sph} with $\delta_0^0<\delta_1^0$. We show that the discrete geodesic curvature takes big values at the middle vertex of $P^0$ (namely $p$). We further plot the discrete curvatures of the circles defined by three consecutive points (black).
	\begin{figure}[H]
		\begin{subfigure}[c]{0.5\linewidth}
			\hspace*{-2cm}
			\includegraphics[width=1.2\linewidth]{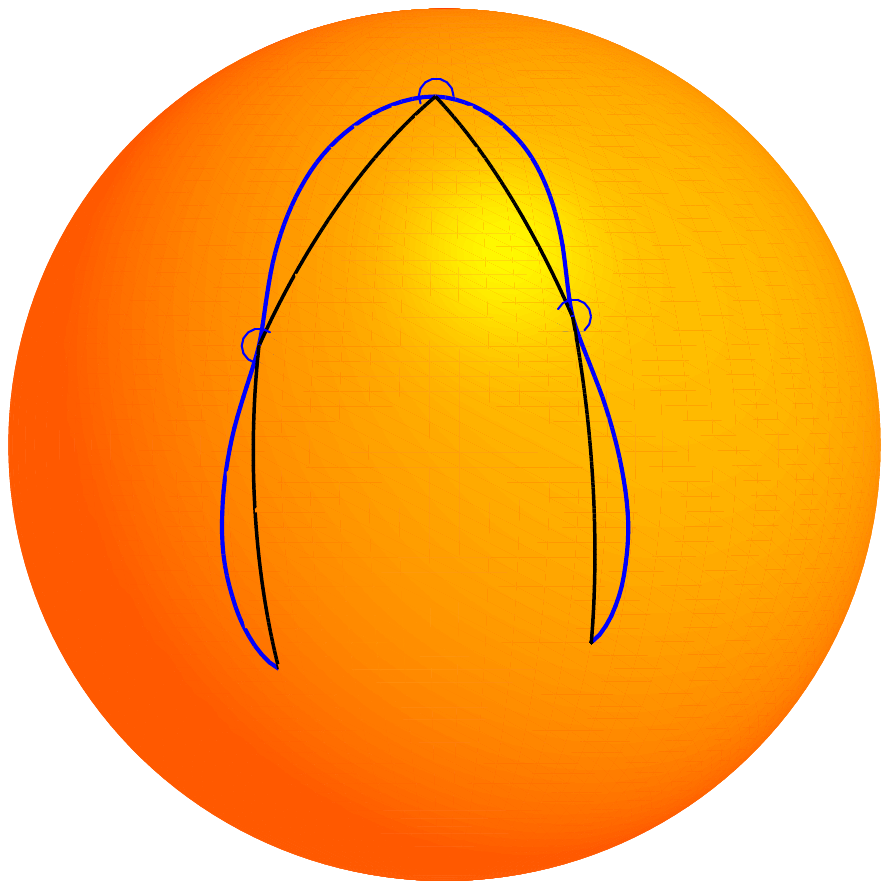}
			\vspace*{-5cm}
			\caption{}
			\label{fig:sph4point2}
		\end{subfigure}
		\hfill
		\begin{subfigure}[c]{0.5\linewidth}
			\hspace*{-3cm}
			\includegraphics[width=1.2\linewidth]{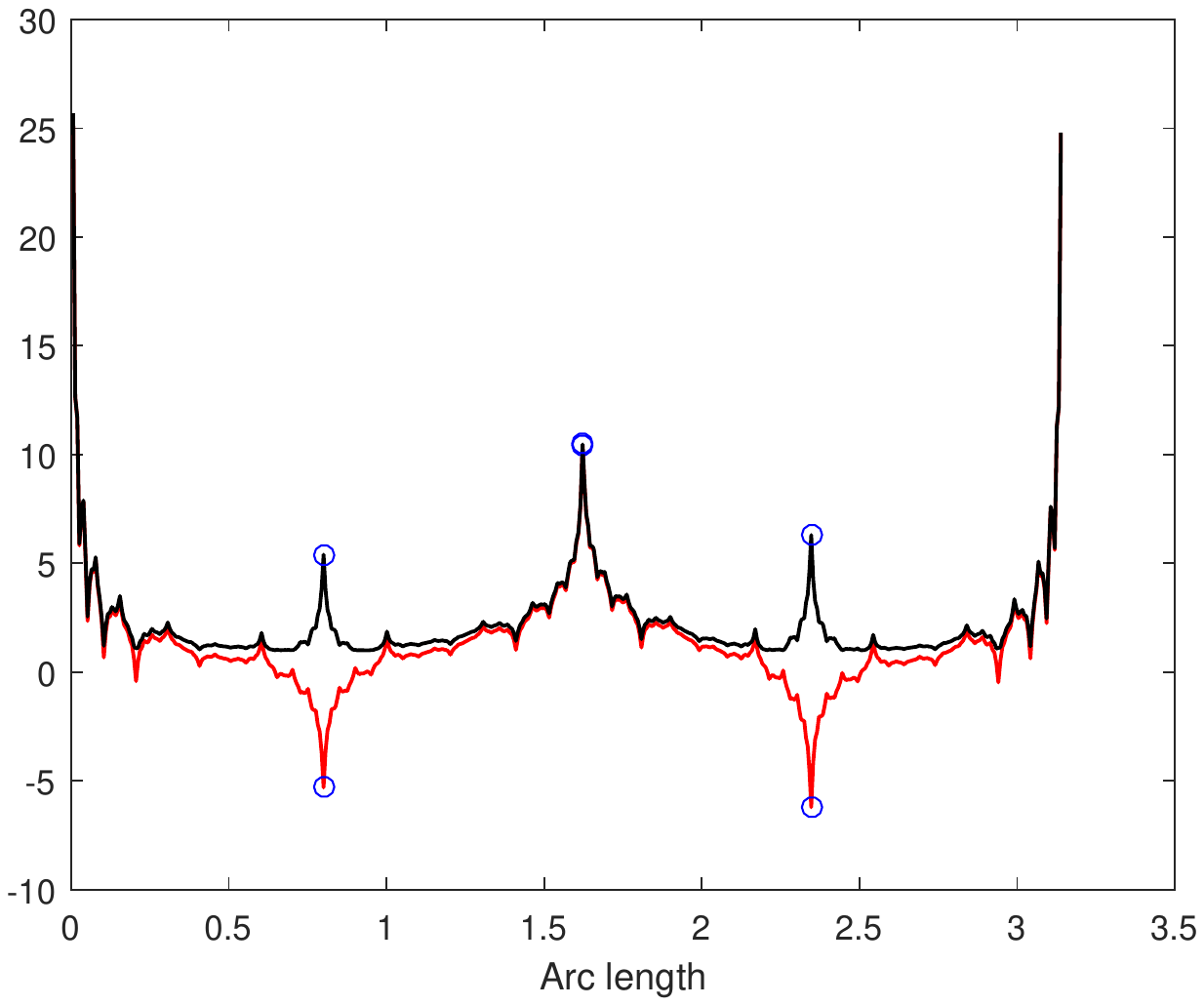}
			\vspace*{-4cm}
			\caption{}
			\label{fig:sph4cur2}
		\end{subfigure}
		\caption{ (a): Limit curve of the 4-point SBIGS-scheme of the polygon $P^0$ with $\delta_0^0>\delta_1^0$. (b). Plots of its: Discrete geodesic curvature (red) and discrete curvature (black) vs arc length  .}
		\label{fig:sph2}
	\end{figure}
	%%%%%%%%%%%%%%%%%%%%%%%%%%%%%%%%%%%%%%%%%%%%
	\begin{figure}[H]
		\vspace*{-5cm}
		\begin{subfigure}[c]{0.5\linewidth}
			\hspace*{-2cm}
			\includegraphics[width=1.2\linewidth]{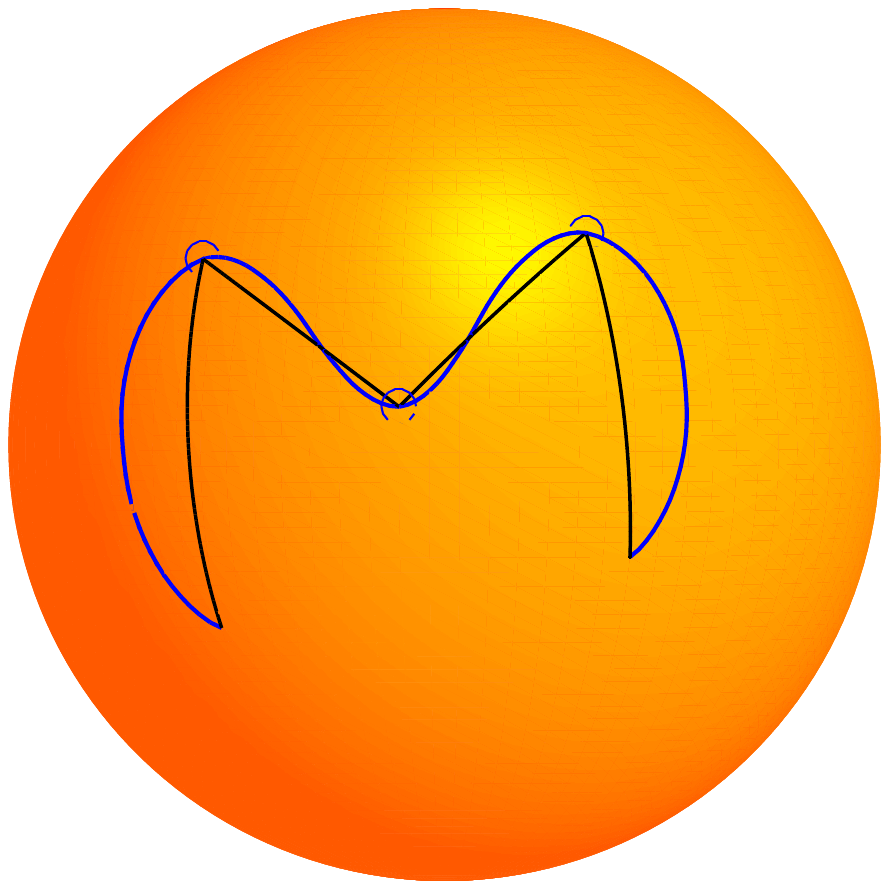}
			\vspace*{-5cm}
			\caption{}
			\label{fig:sph4point}
		\end{subfigure}
		\hfill
		\begin{subfigure}[c]{0.5\linewidth}
			\hspace*{-3cm}
			\includegraphics[width=1.2\linewidth]{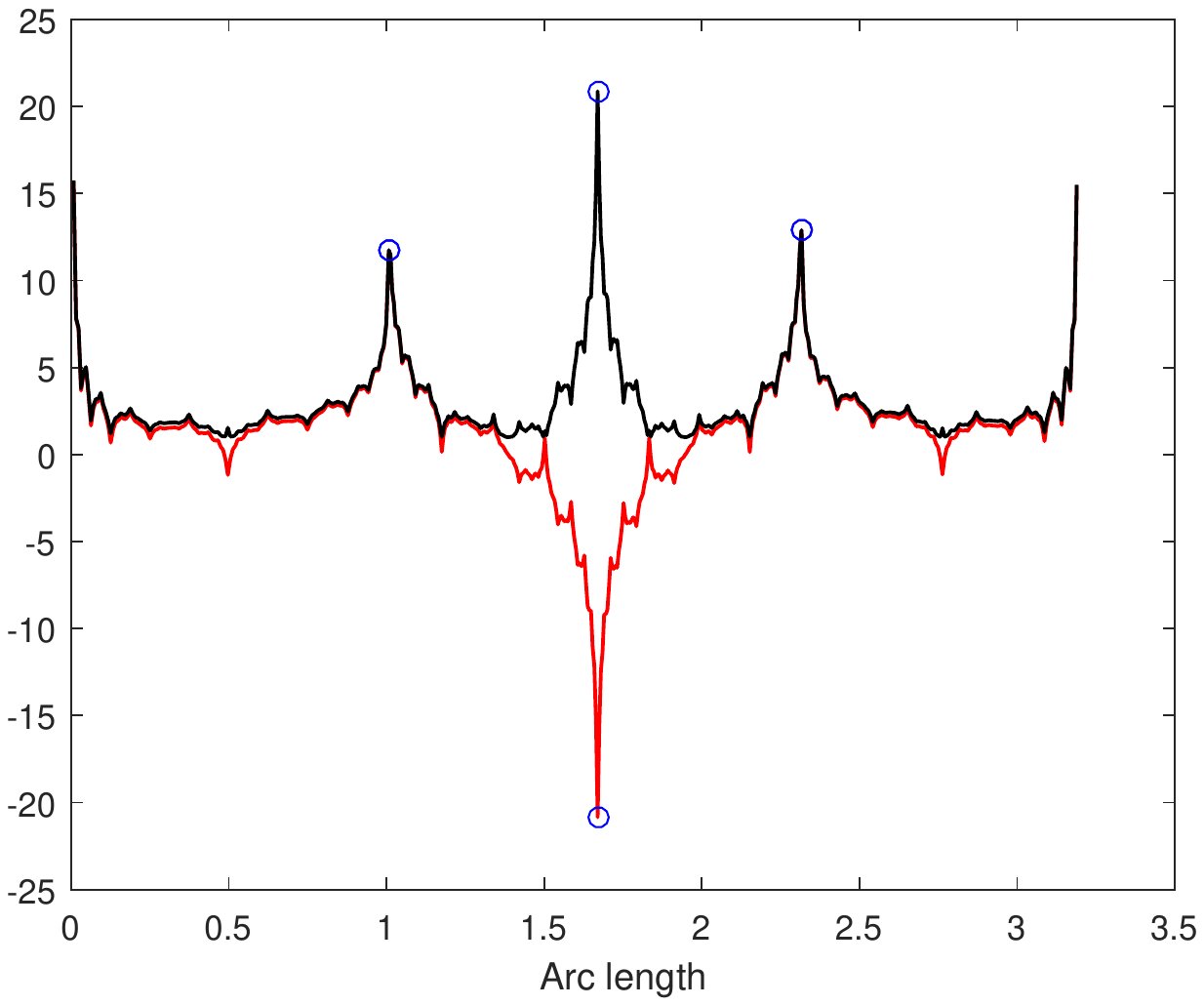}
			\vspace*{-5cm}
			\caption{}
			\label{fig:sph4cur}
		\end{subfigure}
		\caption{
			(a): Limit curve of the 4-point SBIGS-scheme of the polygon $P^0$ with $\delta_0^0<\delta_1^0$.
			(b): Plots of its: Discrete geodesic curvature (red) and discrete curvature (black) vs arc length  .}
		\label{fig:sph}
	\end{figure}
	\begin{figure}[h]
		\hspace*{-1cm}
		\begin{subfigure}[c]{0.5\linewidth}
			\includegraphics[width=1.3\linewidth]{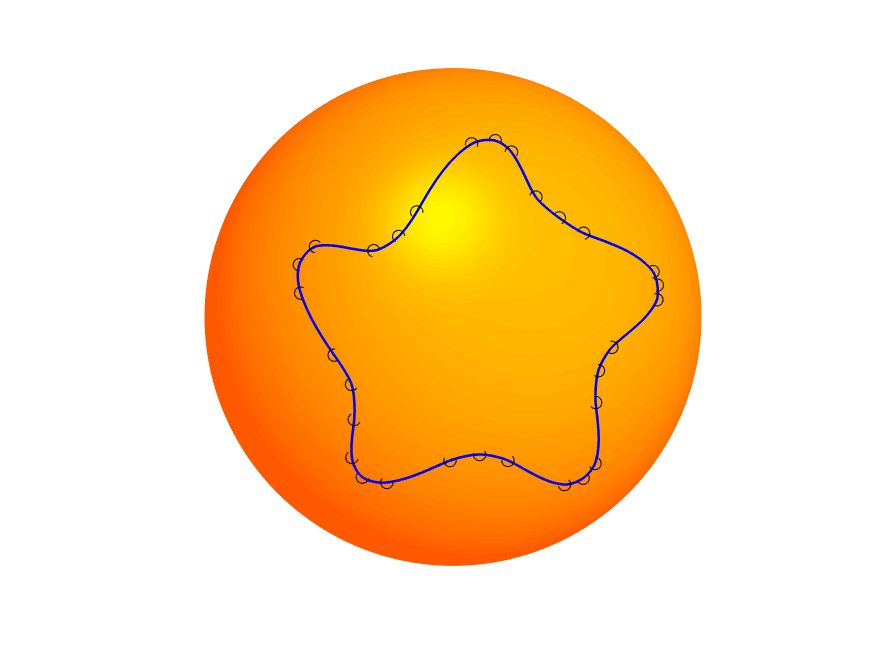}
			\caption{}
			\label{fig:Star}
		\end{subfigure}
		\hfill
		\hspace*{-3cm}
		\begin{subfigure}[c]{0.5\linewidth}
			\includegraphics[width=1.3\linewidth]{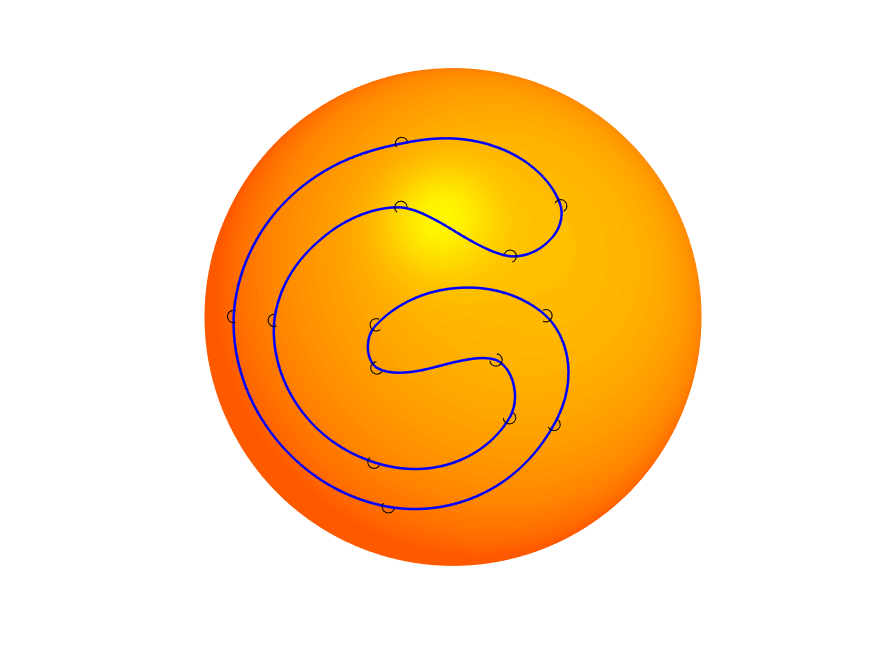}
			\caption{}
			\label{fig:Glike}
		\end{subfigure}
		\caption{Limit curves of the 4-point SBIGS-scheme of a Star-like points (a), G-like points (b).}
		\label{fig:4pt}
	\end{figure}
	Fig. \ref{fig:Star} and Fig. \ref{fig:Glike} show limit curves of the 4-point SBIGS-scheme applied to two different geodesic polygons (a star-like and G-like shapes). The plots show clearly that the curves have no cusps and eventually take the geometry of the control polygons into account.
	\subsection{The curvature-based 6-point spherical subdivision scheme}
	Sabin and Dodgson \cite{Sa} and Cashman et al. \cite{Cas} introduced a scheme able to generates $G^2$-continuous curves on the plane using the curvature of the interpolating circle. The new point $p_{2i+1}^{j+1}$ is determined to be the unique point on the perpendicular bisector of the edge $p_i^j p_{i+1}^j$ such that:
	$$\kappa_{2i+1}^{j+1}=\displaystyle \frac{1}{2}(\kappa_i^j+\kappa_{i+1}^j),$$
	where $\kappa_i^j=\displaystyle \frac{2\, sin(\delta_i^j)}{\parallel p_{i+1}^j-p_{i-1}^j \parallel}$. The new point is then determined by the angle $\delta_{2i+1}^{j+1}$ such that \\
	$sin(\delta_{2i+1}^{j+1})=\displaystyle \frac{\parallel p_{i+1}^j-p_{i-1}^j \parallel}{2}\, \kappa_{2i+1}^{j+1}$.\\
	If we take inspiration from the previous idea to define a spherical generalization using discrete geodesic curvature \eqref{45}, this will fail since from
	$\delta_{2i+1}^{j+1}=\kappa_{2i+1}^{j+1}\, (e_{2i}^{j+1}+e_{2i+1}^{j+1})/2$
	the curvature is not enough to define $p_{2i+1}^{j+1}$. \\
	To solve this problem, we will show in Lemma \ref{c3} that for a convergent and $G^1$-continuous SBIGS-scheme, we can use $\displaystyle \frac{2\, sin(\delta_i^j)}{d(p_{i-1}^j,p_{i+1}^j)}$ as a discrete geodesic curvature at $p_i^j$ instead of \eqref{45}.\\
	If we follow the notations of \eqref{11}-\eqref{15a}, we have:	
	\begin{lem}\label{c3}
		For a convergent SBIGS-scheme, if the sequence $\{ \delta^j\}_j$ is summable, then
		\begin{equation}
			\displaystyle \lim_{\substack{j \to \infty } }\, \frac{2\, sin(\delta_i^j)}{d(p_{i-1}^j,p_{i+1}^j)}=\displaystyle \lim_{\substack{j \to \infty } }\, \kappa_i^j.
			\label{43}
		\end{equation}
	\end{lem}
	\begin{proof}
		This is a simple use of the Darboux formula of spherical triangles:
		$$\displaystyle \frac{cos((\alpha_i^j-\beta_i^j)/2)}{sin(\gamma_i^j/2)}=\frac{sin((e_{i-1}^j+e_i^j)/2)}{sin(d(p_{i-1}^j,p_{i+1}^j)/2)},$$
		with $\gamma_i^j=\pi-\delta_i^j$ the internal angle at $p_i^j$.
		Then
		$$\displaystyle 2\, sin(\delta_i^j/2)\, \frac{cos(\delta_i^j/2)}{sin(d(p_{i-1}^j,p_{i+1}^j)/2)}=2\, sin(\delta_i^j/2)\, \frac{cos((\alpha_i^j-\beta_i^j)/2)}{sin((e_{i-1}^j+e_i^j)/2)} .$$
		Hence
		$$\displaystyle \frac{sin(\delta_i^j)}{sin(d(p_{i-1}^j,p_{i+1}^j)/2)}=2\, \frac{sin(\delta_i^j/2)}{sin((e_{i-1}^j+e_i^j)/2)}\, cos((\alpha_i^j-\beta_i^j)/2) .$$
		Since $\displaystyle \kappa_i^j:=\frac{2\, \delta_i^j}{e_{i-1}^j+e_i^j}$, we obtain
		$$		\displaystyle \frac{sin(\delta_i^j)}{sin(d(p_{i-1}^j,p_{i+1}^j)/2)}=\kappa_i^j\, \frac{\displaystyle sin(\delta_i^j/2)}{\delta_i^j}\,  \frac{e_{i-1}^j+e_i^j}{sin((e_{i-1}^j+e_i^j)/2)}\, cos\Big( \displaystyle \frac{\alpha_i^j-\beta_i^j}{2} \Big).$$
		By Proposition 4.4 of \cite{Be}, we have $\displaystyle \lim_{j \to \infty} e_i^j=\lim_{j \to \infty}\, \alpha_i^j=\lim_{j \to \infty}\, \beta_i^j=0$. Finally, we deduce the claim.

	\end{proof}
	Let $\kappa_i^j=\displaystyle \frac{2\, sin(\delta_i^j)}{d(p_{i-1}^j,p_{i+1}^j)}$.	We define the curvature-based 6-point spherical scheme as a SBIGS-scheme in which the new point is determined using its discrete geodesic curvature. We define the new discrete geodesic curvature \\
	$\kappa_{2i+1}^{j+1}= \frac{2\, sin(\delta_{2i+1}^{j+1})}{d(p_{i-1}^j,p_{i+1}^j)}$ by applying twice the odd sub-mask of second-ordre differences of the linear 6-point scheme \\
	$\frac{1}{64}\, \{-3,19,19,-3\}$
	to $\kappa_{i-1}^{j}, \, \kappa_{i}^{j},\, \kappa_{i+1}^{j}$ and $\kappa_{i+2}^{j}$, namely:
	$$\kappa_{2 i+1}^{j+1}=-\frac{3}{32} \kappa_{i-1}^{j}+\frac{19}{32} \kappa_{i}^{j}+\frac{19}{32} \kappa_{i+1}^{j}-\frac{3}{32} \kappa_{i+2}^{j}.$$
	In order to have $\displaystyle \frac{d( p_i^j,p_{i+1}^j)}{2}\, \kappa_{2i+1}^{j+1}=sin(\delta_{2i+1}^{j+1})$ in the interval $\left[-1,1\right]$, we apply some iterations of the spherical angle-based 4-point scheme of section 3.1 to the initial polygon.\\
	
	On the triangle formed by $p_i^j$, $p_{2i+1}^{j+1}$ and $c_{i,i+1}^j(l_{i,i+1}^j/2)$ (see Fig. \ref{fig:f11}), we use the right spherical triangle formula
	$$sin(\alpha_{2i+1}^{j+1})=\displaystyle \frac{cos\Big((\pi-\delta_{2i+1}^{j+1})/2\Big)}{cos(d(p_{i-1}^j,p_{i+1}^j)/2)},$$
	to compute $\alpha_{2i+1}^{j+1}$. Finally we use \eqref{15a} to obtain the new point $p_{2i+1}^{j+1}$.\\
	
	Fig. \ref{fig:incenter} shows several experiments of the proposed SBIGS-scheme. We observe that the scheme is convergent and $G^1$-continuous. From Corollary \ref{co1}, we have to prove that the sequence $\{\delta^j\}_j$ is summable. In Fig. \ref{fig:RatioDelt}, we display the ratio  $\delta^{j+1}/\delta^j$ depending on j for each initial geodesic polygon considered. The ratios are less than 1, then it means that the corresponding sequence $\{\delta^j\}_j$ behaves like a geometric sequence and is summable.\\
	%%%%%%%%%%%%%%%%%%%%%%%%%%%%%%%%%%%%%%%%%%%%%%%%%%%%%
	%%%%%%%%%%%%%%%%%%%%%%%%%%%%%%%%%%%%%%%%%%%%%%%%%%%%%%%%%%%%%%%%
	%%%%%%%%%%%%%%%%%%%%%%%%%%%%%%%%%%%%%%%%%%%%%%%%%%%%%%%%%%%%%%%%%%%%%%%%%

	%%%%%%%%%%%%%%%%%%%%%%%%%%%%%%%%%%%%%%%%%%%%%%%%%%%%%%%
	%%%%%%%%%%%%%%%%%%%%%%%%%%%%%%%%%%%%%%%%%%%%%%%%%%%%%%%%%%%%%%%%%%%%%%%
	
	\begin{figure}[H]
		\begin{subfigure}[c][0.6\width]{0.4\textwidth}
			\centering
			\hspace*{-0.2cm}
			\includegraphics[width=0.9\linewidth]{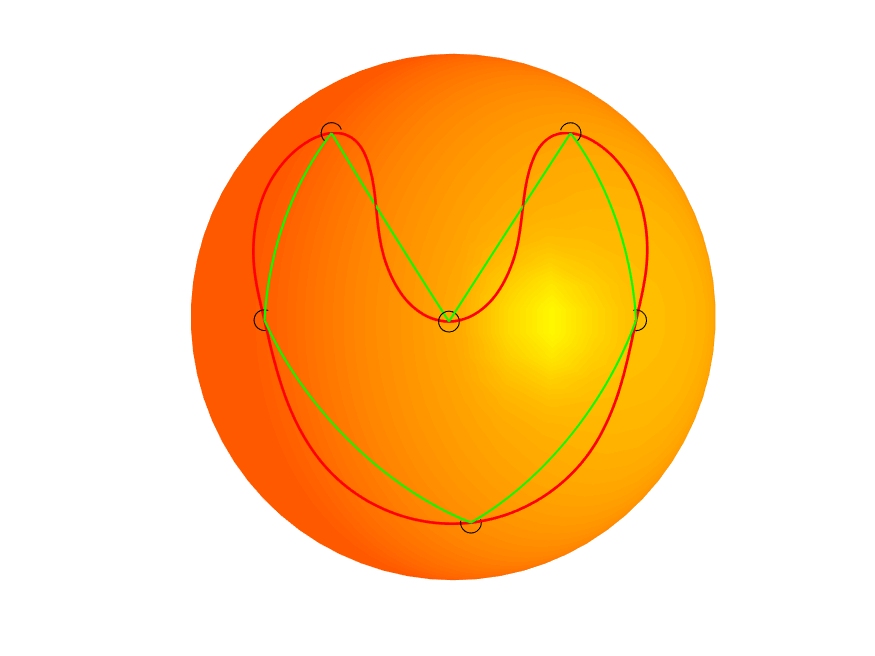}
			\vspace*{0cm}
			\caption{}
			\label{fig:Ex2}
		\end{subfigure}
		\hspace{-2.2cm}
		\hfill
		\begin{subfigure}[c][0.6\width]{0.4\textwidth}
			\centering
			\includegraphics[width=0.9\linewidth]{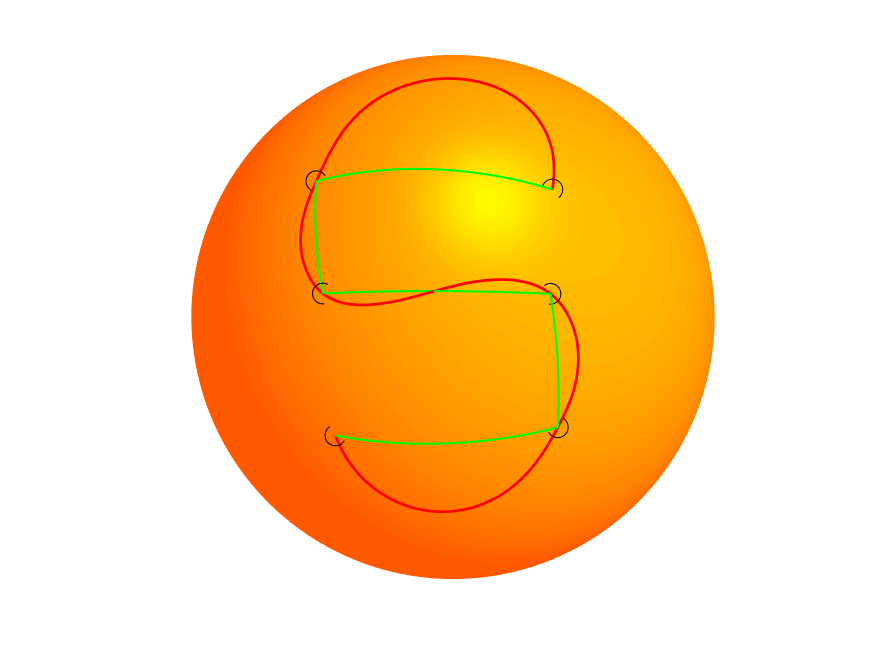}
			\vspace*{0cm}
			\caption{}
			\label{fig:Slike}
		\end{subfigure}
		\hspace{-2cm}
		\hfill
		\begin{subfigure}[c][0.6\width]{0.4\textwidth}
			\centering
			\includegraphics[width=0.9\linewidth]{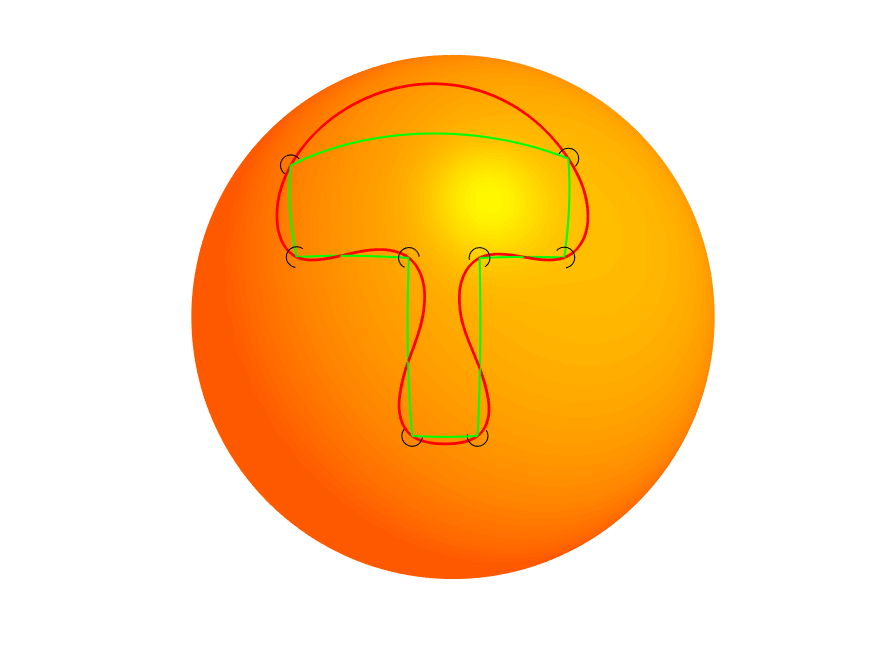}
			\vspace*{0cm}
			\caption{}
			\label{fig:Tlike}
		\end{subfigure}
		\vspace*{1cm}
		\caption{Limit curves of the curvature-based 6-point spherical scheme: M-like (a), S-like (b) and T-like (c) curves.}
\label{fig:incenter}
	\end{figure}
	
	%%%%%%%%%%%%%%%%%%%%%%%%%%%%%%%%%%%%%%%%%%%%%%%%%%%

		\begin{figure}[H]
		\hspace*{-3cm}
		\begin{subfigure}[c][0.5\width]{0.4\textwidth}
			\includegraphics[width=1.4\linewidth]{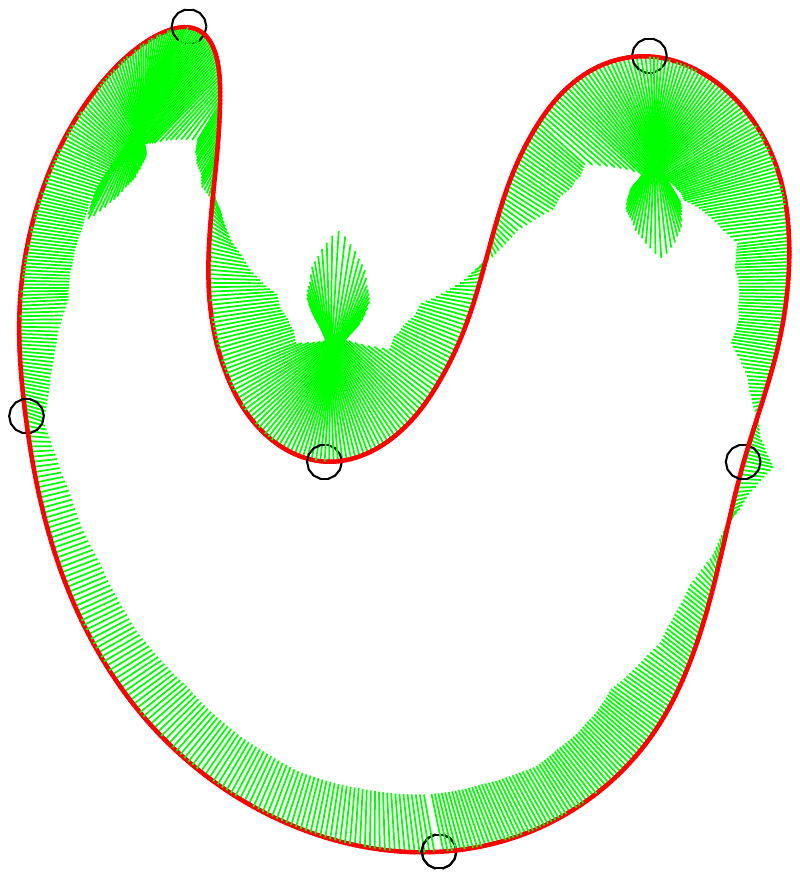}
			\vspace*{-5cm}
			\caption{}
			\label{fig:}
		\end{subfigure}
		\hfill
		\hspace*{-3cm}	
		\begin{subfigure}[c][0.5\width]{0.4\textwidth}
			\includegraphics[width=1.4\linewidth]{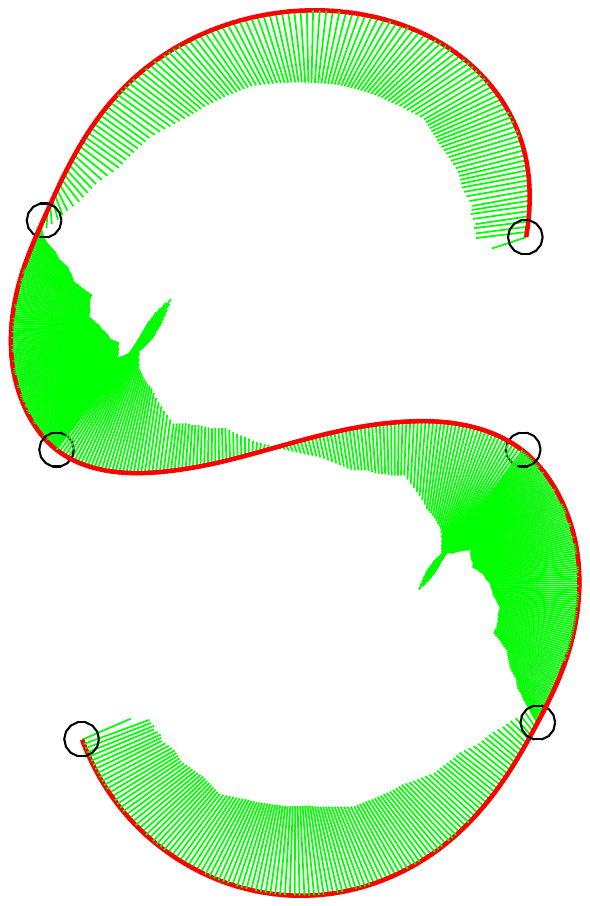}
			\vspace*{-5cm}
			\caption{}
			\label{fig:}
			
		\end{subfigure}
		\hfill
		\hspace*{-3cm}
		\begin{subfigure}[c][0.5\width]{0.4\textwidth}
			\includegraphics[width=1.4\linewidth]{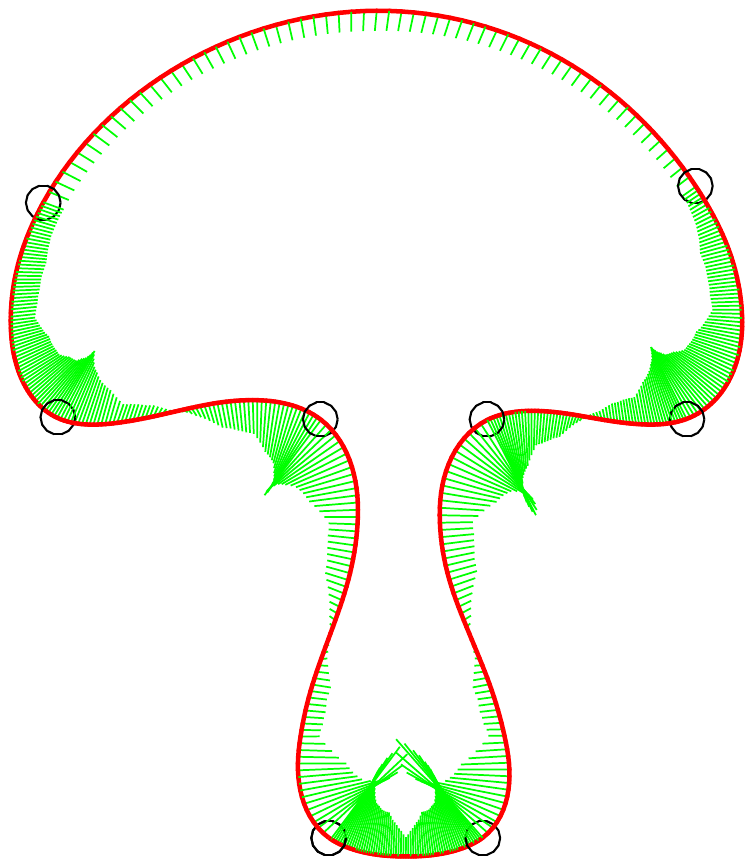}
			\vspace*{-5cm}
			\caption{}
			\label{fig:}
		\end{subfigure}
		\vspace*{4cm}
		\caption{Discrete normal lines of the curvature-based 6-point spherical scheme: M-like (a), S-like (b) and T-like (c) curves.}
\label{fig:Dcurva6pt}
	\end{figure}
	
	%%%%%%%%%%%%%%%%%%%%%%%%%%%%%%%%%%%%%%%%%%%%%%
	\vspace*{-1cm}
			\begin{figure}[H]
		\hspace*{-3cm}
		\begin{subfigure}[c][0.5\width]{0.4\textwidth}
			\includegraphics[width=1.2\linewidth]{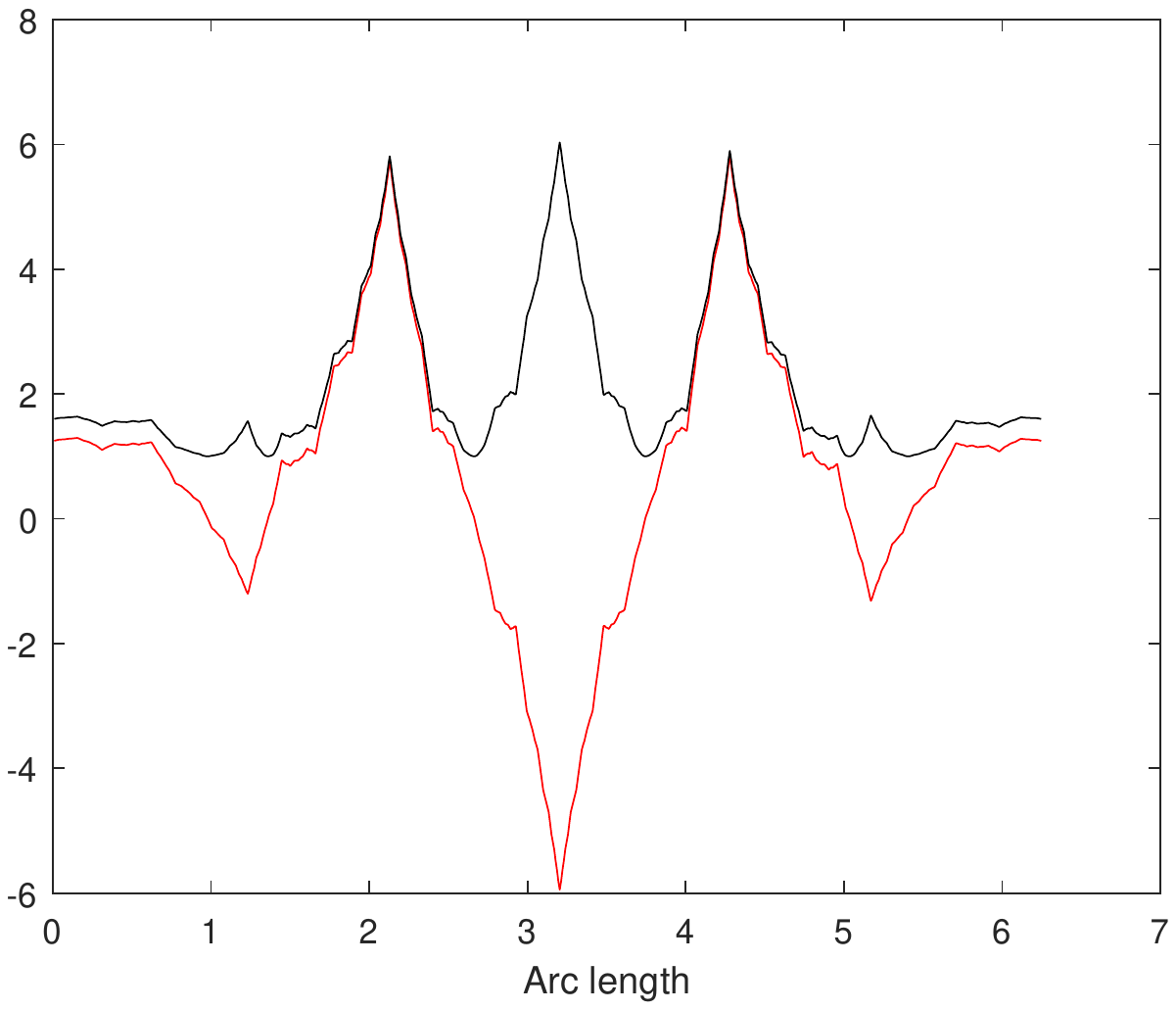}
			\vspace*{-4cm}
			\caption{}
\label{fig:Ex2cur}
		\end{subfigure}
		\hfill
		\hspace*{-3cm}	
		\begin{subfigure}[c][0.5\width]{0.4\textwidth}
			\includegraphics[width=1.2\linewidth]{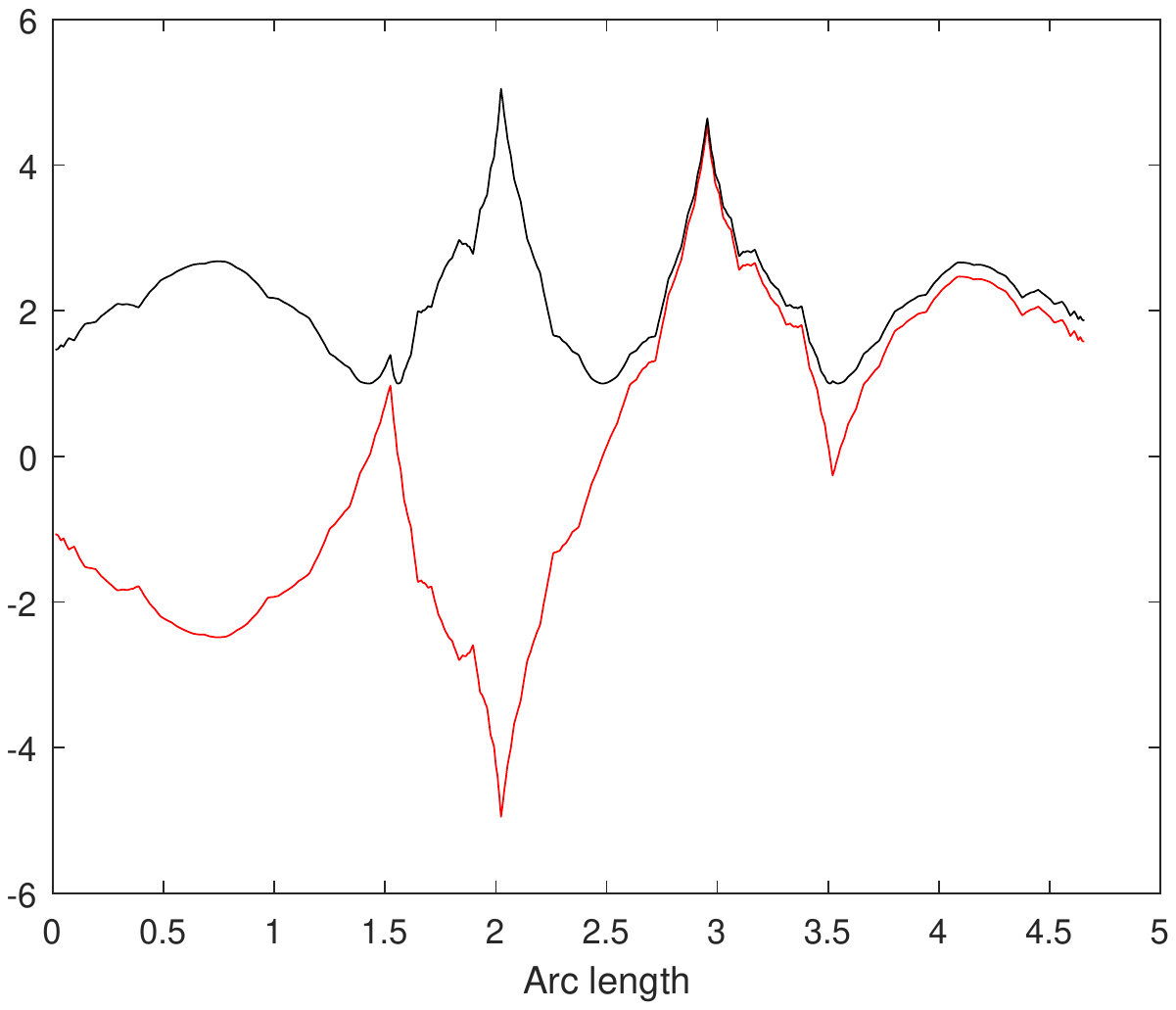}
			\vspace*{-4cm}
			\caption{}
			\label{fig:Slikecur}
			
		\end{subfigure}
		\hfill
		\hspace*{-3cm}
		\begin{subfigure}[c][0.5\width]{0.4\textwidth}
			\includegraphics[width=1.2\linewidth]{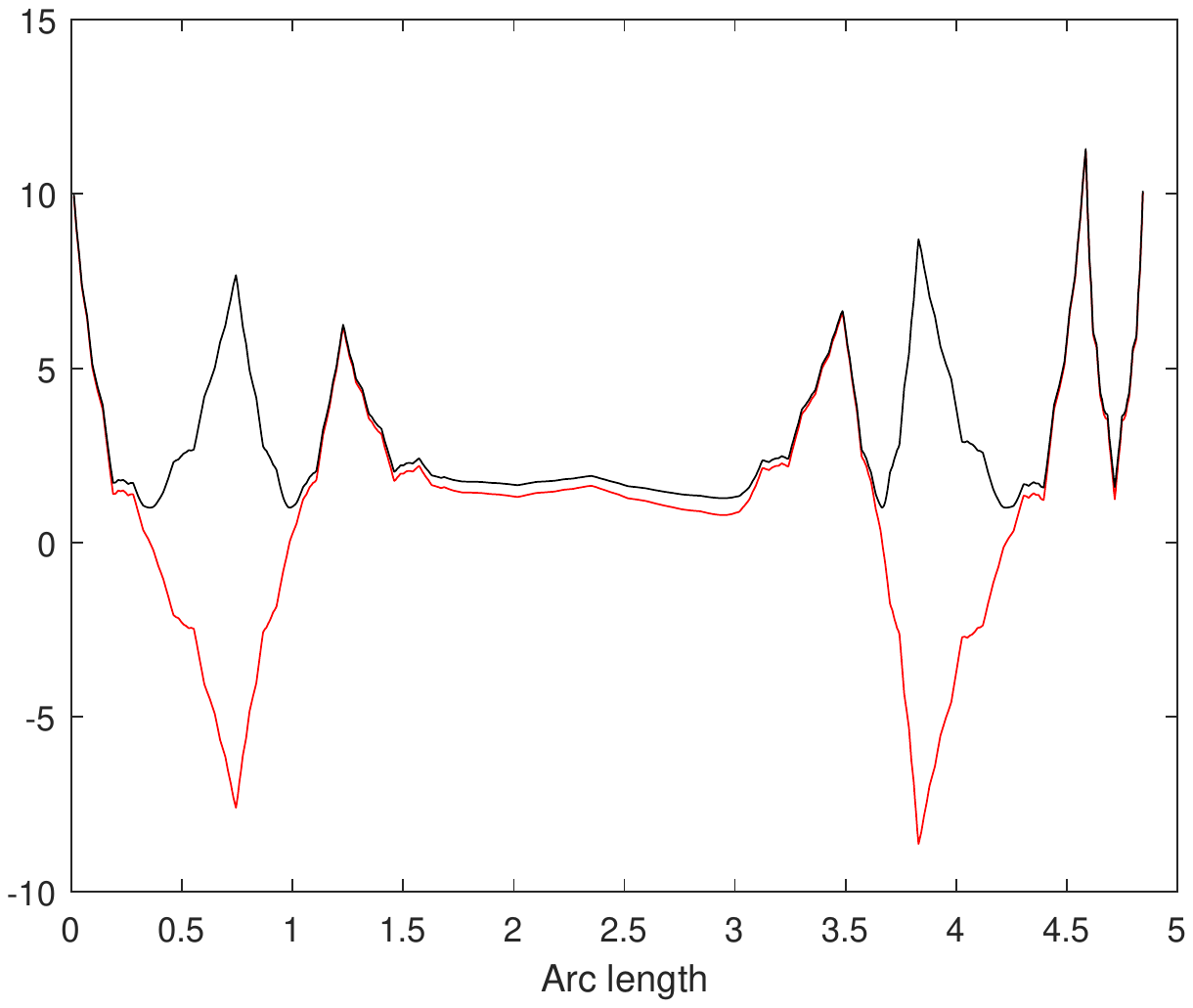}
			\vspace*{-4cm}
			\caption{}
			\label{fig:Tlikecur}
		\end{subfigure}
		\vspace*{3cm}
		\caption{Discrete geodesic curvature (red) and discrete curvature (black) plots vs arc length  of the M-like (a), S-like  (b) and T-like (c) curves.}
\label{fig:incentercur}
	\end{figure}
	
In order to test the $G^2$-continuity of examples in Fig. \ref{fig:incenter}, we display in Fig. \ref{fig:incentercur} the values of the discrete geodesic curvature (red) and the discrete curvature (black) of the circles $C_i^j$ defined by three consecutive points $p_{i-1}^j$, $p_{i}^j$ and $p_{i+1}^j$  for geodesic polygons at iteration 9. We see that both curvatures are continuous functions. We plot in Fig. \ref{fig:Dcurva6pt} the discrete normals of $C_i^j$ (scaled) of the considered examples. We see that the discrete normal vectors form a continuous field along the curves.\\
Motivated by the idea of considering the summability of the sequence $\{ \delta^j\}_j$ to prove the $G^1$-continuity of plane curves,  Volontè, E. \cite{Vo} suggests to study the sequence of maximum curvature differences: $\nabla \kappa^j:=\underset{i\in \mathbb{Z}}{sup}\, |\frac{1}{r_{i+1}^{j}}-\frac{1}{r_i^j}|$, where $r_i^j$ is the radius of the circle $C_i^j$,  in order to guarantee the $G^2$-continuity of plane curves. The proof is carefully explained but some lemmas are missing . However, the author shows by numerical evidences that the statements of the missing lemmas are true.  \\
Inspiring by the previous ideas, we propose to use the discrete geodesic curvature to argue the $G^2$-continuity of spherical curves. In Fig. \ref{fig:MaxDiffCur} and Fig. \ref{fig:RatioMaxDiff} we compute the sequence $\nabla \kappa^j:=\underset{i\in \mathbb{Z}}{sup}\, |\kappa_{i+1}^{j}-\kappa_i^j|$ of the maximum curvature differences and the ratio  $\nabla \kappa^{j+1}/\nabla \kappa^j$ of limit curves in examples Fig. \ref{fig:incenter}. We see that the difference between the discrete geodesic curvatures $\kappa_i^j$ and $\kappa_{i+1}^j$ decays to 0 and the ratios run to a constant strictly less than 1.
		\begin{figure}[H]
\hspace*{-3cm}
		\begin{subfigure}[c][0.5\width]{0.4\textwidth}
		\includegraphics[width=1.4\linewidth]{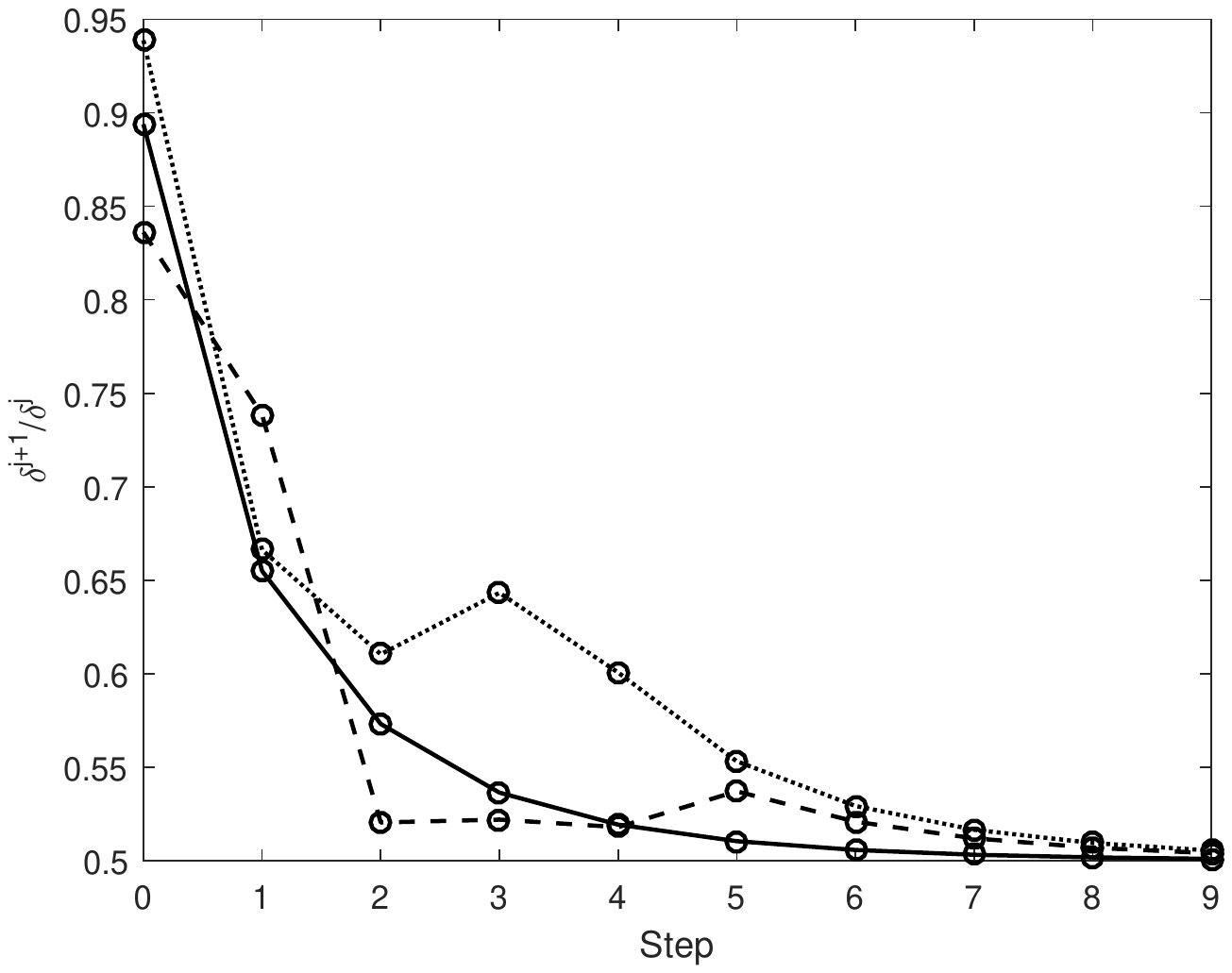}
		\vspace*{-4cm}
		\caption{}
        \label{fig:RatioDelt}
		\end{subfigure}
\hfill
\hspace*{-3cm}	
		\begin{subfigure}[c][0.5\width]{0.4\textwidth}
		\includegraphics[width=1.4\linewidth]{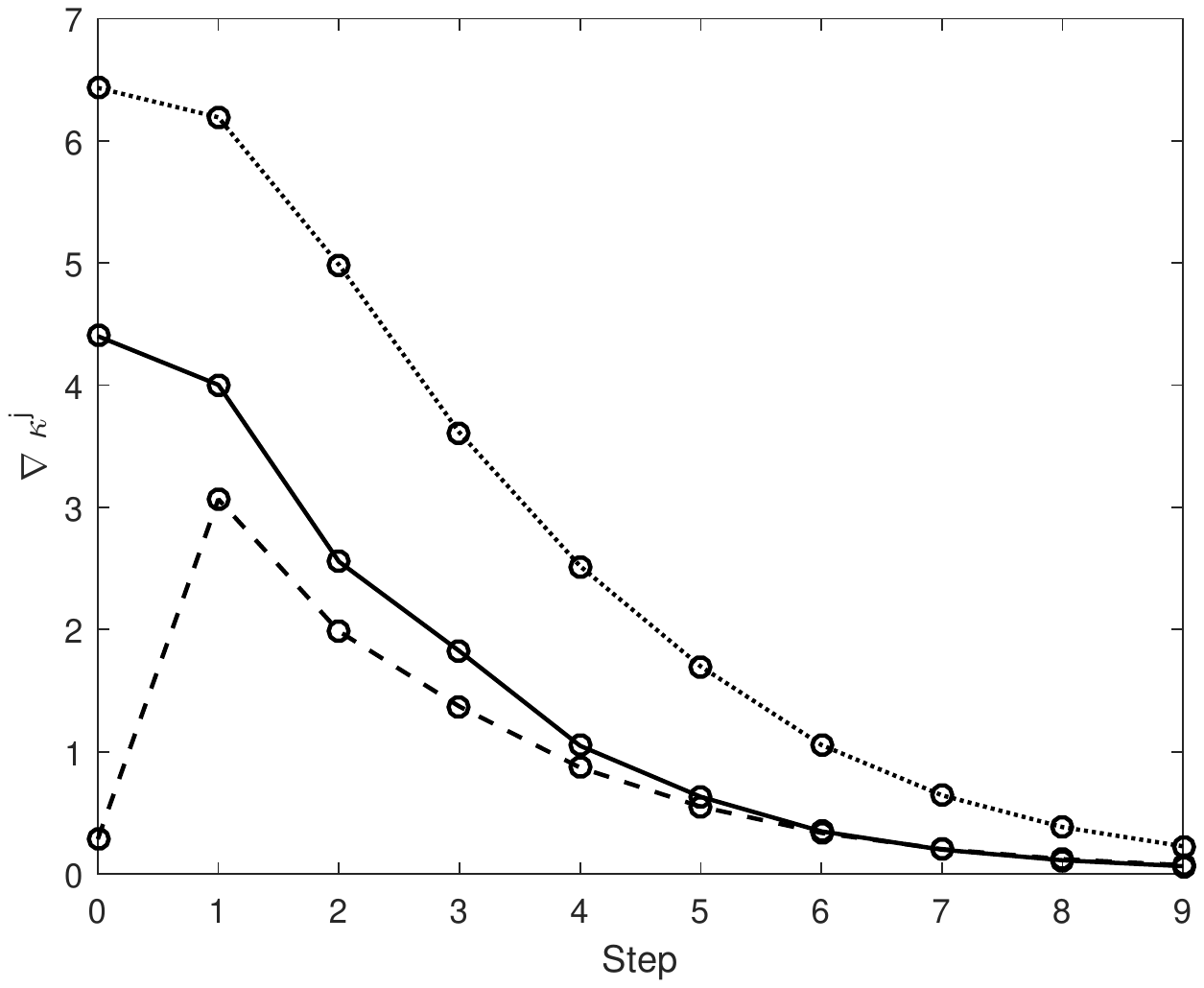}
		\vspace*{-4cm}
		\caption{}
		\label{fig:MaxDiffCur}

		\end{subfigure}
\hfill
\hspace*{-3cm}
	\begin{subfigure}[c][0.5\width]{0.4\textwidth}
		\includegraphics[width=1.4\linewidth]{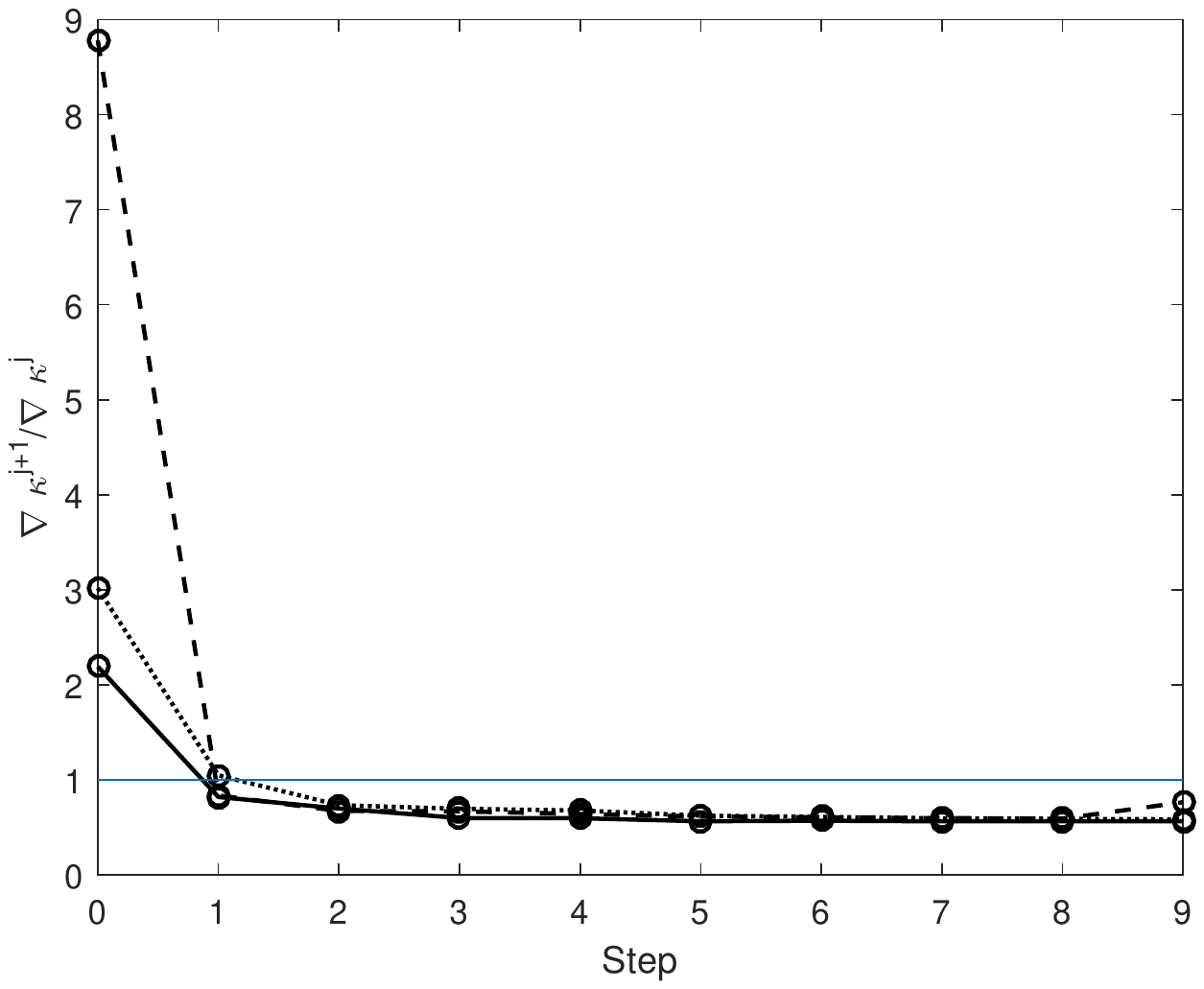}
		\vspace*{-4cm}
		\caption{}
		\label{fig:RatioMaxDiff}
	\end{subfigure}
\vspace*{4cm}
		\caption{(a): Ratio of decay $\delta^{j+1}/\delta^j$ of the sequence $\{\delta^j\}_j$, (b): Maximum of curvature differences $\nabla \kappa^j$, (c): Ratio of decay $\nabla \kappa^{j+1}/\nabla \kappa^j$ of examples Fig. \ref{fig:incenter}: M-like (solid line), T-like (dashed line) and S-like (dotted line).}
		\label{}
	\end{figure}

	\section{Conclusion and future work}
	We have seen that the introduced discrete geodesic curvature is more practical to give examples of spherical  $G^2$-continuous curves.. On the other hand, we gave a spherical generalization of the planar angle-based 4-point subdivision scheme interpolating unit vectors. A scheme generating $G^2$-continuous curves is also proposed. In the future work, we will investigate the notion of  discrete geodesic curvature to study the shape warping and morphing of curves on surfaces.

	\appendix
	\section{}
	Let $P^j:=\{p_i^j \, \in \mathbb{R}^2 \}$ and
	$$\alpha_{i}^{j}:=\sphericalangle(\overrightarrow{p_{i-1}^{j}p_{i+1}^{j}},\overrightarrow{p_{i-1}^{j}p_{i}^{j}}), \qquad \beta_{i}^{j}:=\sphericalangle(\overrightarrow{p_{i}^{j}p_{i+1}^{j}},\overrightarrow{p_{i-1}^{j}p_{i+1}^{j}}), \qquad \delta_{i}^{j}:=\sphericalangle (\overrightarrow{p_{i}^{j}p_{i+1}^{j}},\overrightarrow{p_{i-1}^{j}p_{i}^{j}}).$$
	The planar angle-based 4-point scheme \cite{Dy} is  defined by:
	\begin{equation}
		\left\{
		\begin{aligned}
			&p_{2i}^{j+1}=p_{i}^{j},  \\
			&p_{2i+1}^{j+1}=\displaystyle p_i^j+\frac{1}{2cos(\alpha_{2i+1}^{j+1})}\, R(\alpha_{2i+1}^{j+1})\, (p_{i+1}^j-p_i^j),\\
		\end{aligned}
		\right.
		\label{40}
	\end{equation}
	where $$\alpha_{2i+1}^{j+1}=\beta_{2i+1}^{j+1}=\displaystyle \frac{\delta_{i}^{j}+\delta_{i+1}^{j}}{8},$$
	and $R(\alpha_{2i+1}^{j+1})$ is the rotation matrix by angle $\alpha_{2i+1}^{j+1}$. We have:
	\begin{prop}\label{t1}
		Let $P^0=\{p_{-2},p_{-1},p,p_1,p_2\}$ be  a planar polygon with equal edge lengths. If $\delta_{-1}^0=\delta_1^0$ and $\delta_{0}^0 \neq \delta_{1}^0$ (see Fig. \ref{fig:p4}), then  $\displaystyle \lim_{j \to \infty}\, \frac{\delta_0^{j}}{e_0^{j}}=\pm \infty$.
	\end{prop}
	\begin{proof}
		\begin{figure}
			\centering
			\includegraphics[width=0.7\linewidth]{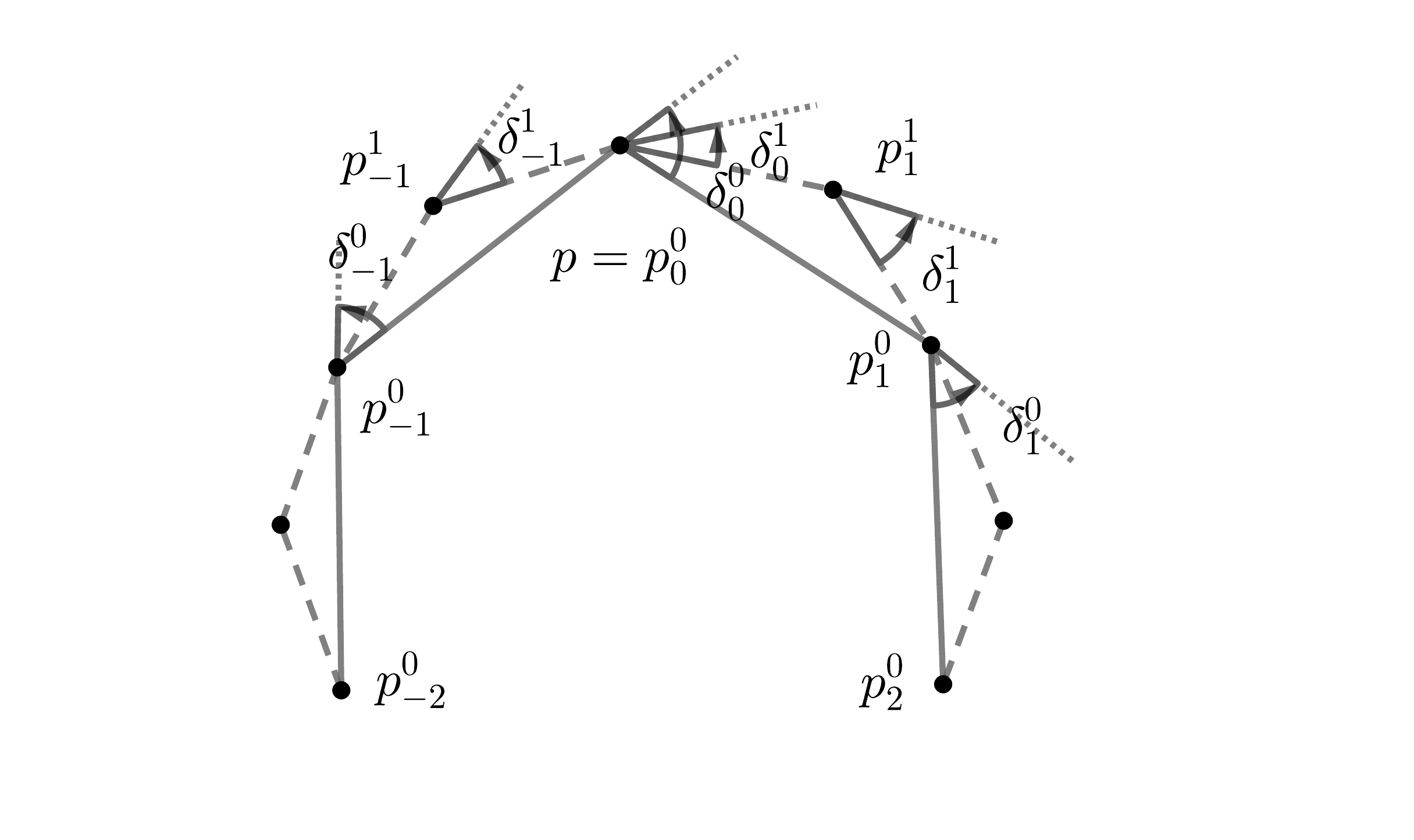}
			\caption{A first iteration of the polygon $P^0$.}
			\label{fig:p4}
		\end{figure}
		In the planar case, Eqs. \eqref{38} become
		\begin{equation}
			\left\{
			\begin{aligned}
				&\delta_{2i-1}^{j+1}=\displaystyle \frac{\delta_{i-1}^j+\delta_{i}^j}{4},\\
				&\delta_{2i}^{j+1}=\frac{-\delta_{i-1}^j+6\delta_i^j-\delta_{i+1}^j}{8},\\
				&\delta_{2i+1}^{j+1}=\displaystyle \frac{\delta_i^j+\delta_{i+1}^j}{4}.\\
			\end{aligned}
			\right.
		\end{equation}
		If  we set  $A=\begin{pmatrix}
			2&2&0\\
			-1&6&-1\\
			0&2&2\\
		\end{pmatrix}$, then the above system becomes
		$$\displaystyle
		\begin{pmatrix}
			\delta_{2i-1}^{j+1}\\
			\delta_{2i}^{j+1}\\
			\delta_{2i+1}^{j+1}\\
		\end{pmatrix}=\frac{1}{8} A
		\begin{pmatrix}
			\delta_{i-1}^{j}\\
			\delta_{i}^{j}\\
			\delta_{i+1}^j\\
		\end{pmatrix}.$$
		
		Without losing generality, we can suppose that $i=0$. We get
		
		$$\displaystyle
		\begin{pmatrix}
			\delta_{-1}^{j+1}\\
			\delta_{0}^{j+1}\\
			\delta_{1}^{j+1}\\
		\end{pmatrix}=\frac{1}{8} A
		\begin{pmatrix}
			\delta_{-1}^{j}\\
			\delta_{0}^{j}\\
			\delta_{1}^j\\
		\end{pmatrix}.$$
		Hence
		$$\displaystyle
		\begin{pmatrix}
			\delta_{-1}^{j}\\
			\delta_{0}^{j}\\
			\delta_{1}^{j}\\
		\end{pmatrix}=\frac{1}{2^{3j}} A^j
		\begin{pmatrix}
			\delta_{-1}^{0}\\
			\delta_{0}^{0}\\
			\delta_{1}^0\\
		\end{pmatrix}.$$
		Finally we obtain
		\begin{equation}
			\left\{
			\begin{aligned}
				\displaystyle
				&\delta_{-1}^j=\Big( \frac{-j}{2^{j+2}}+\frac{1}{2^{j+1}}+\frac{1}{2^{2j+1}}\Big) \delta_{-1}^0+\frac{j}{2^{j+1}}\delta_{0}^0+ \Big( \frac{-j}{2^{j+2}}+\frac{1}{2^{j+1}}- \frac{1}{2^{2j+1}} \Big)\delta_{1}^1,\\
				&\delta_0^j=\frac{-j}{2^{j+2}}\delta_{-1}^0+\frac{6}{2^{j+2}}\delta_{0}^0-\frac{-j}{2^{j+2}}\delta_{1}^0       ,\\
				&\delta_1^j=\Big( \frac{-j}{2^{j+2}}+\frac{1}{2^{j+1}}-\frac{1}{2^{2j+1}}\Big) \delta_{-1}^0+\frac{j}{2^{j+1}}\delta_{0}^0+ \Big( \frac{-j}{2^{j+2}}+\frac{1}{2^{j+1}}+ \frac{1}{2^{2j+1}} \Big)\delta_{1}^1 .\\
			\end{aligned}
			\right.
			\label{20}
		\end{equation}
		Since $\delta_{-1}^0=\delta_1^0$, Eqs. \eqref{20} become
		\begin{equation*}
			\left\{
			\begin{aligned}
				&\delta_0^j=\displaystyle \frac{2+j}{2^{j+1}} \delta_0^0-\frac{j}{2^{j+1}}\delta_1^0,\\
				&\delta_1^j= \frac{j}{2^{j+1}} \delta_0^0+\frac{2-j}{2^{j+1}}\delta_1^0.\\
			\end{aligned}
			\right.
		\end{equation*}
		On the other hand, we have $\alpha_1^j=\alpha_{-1}^j=\displaystyle \frac{\delta_0^{j-1}+\delta_1^{j-1}}{8}$, then
		\begin{equation}
			\alpha_1^j=\displaystyle  \frac{j(\delta_0^0-\delta_1^0)}{2^{j+2}}+\frac{2}{2^{j+2}}\delta_1^0.
			\label{19}
		\end{equation}
		Let $e_0^j:=\|p_1^j-p\|=\|p_{-1}^j-p\|$. Since $e_0^{j+1}=\displaystyle \frac{e_0^j}{2\, cos(\alpha_1^{j+1})}$, this yields
		$$e_0^j=\displaystyle \frac{e_0^0}{2^j} \prod_{k=1}^{j} cos(\alpha_1^k)^{-1}.$$
		Finally
		\begin{align}
			\displaystyle \frac{ \delta_0^j}{ e_0^j}&=\frac{\displaystyle  \frac{2+j}{2^{j+1}} \delta_0^0-\frac{j}{2^{j+1}}\delta_1^0    }{\displaystyle \frac{e_0^0}{2^j} \prod_{k=1}^{j} cos(\alpha_1^k)^{-1}}\\
			&=\displaystyle \frac{ j(\delta_0^0-\delta_1^0) +2\delta_0^0 }{2e_0^0\, \displaystyle \prod_{k=1}^{j} cos(\alpha_1^k)^{-1}}.
			\label{35}
		\end{align}
		By \eqref{19}, the series $\displaystyle \sum_{j}\, |\alpha_1^j|$ is convergent, then $\displaystyle \sum_{j}\, (\alpha_1^j)^2$ is also convergent. By equivalence, the product\\ $\displaystyle \prod_{k=1}^{j} \Bigg( 1-\frac{(\alpha_1^k)^2}{8}\Bigg)^{-1}$ is also convergent and so is $\displaystyle \prod_{k=1}^{j} cos(\alpha_1^k)^{-1}$. Finally, $\displaystyle \lim_{j \to \infty}\, \frac{\delta_0^{j}}{e_0^{j}}=\pm \infty$  once $\delta_{0}^0 \neq \delta_{1}^0$  (depends on the sign of $(\delta_0^0-\delta_1^0)$).
	\end{proof}
	
	%% main text
	
	%% The Appendices part is started with the command \appendix;
	%% appendix sections are then done as normal sections
	%% \appendix
	
	%% \section{}
	%% \label{}
	
	%% If you have bibdatabase file and want bibtex to generate the
	%% bibitems, please use
	%%
	%  \bibliographystyle{elsarticle-harv}
	%%  \bibliography{<your bibdatabase>}

\begin{thebibliography}{9}
		
		\bibitem{Be}  Bellaihou, M., Ikemakhen, A., 2020. {\it   Spherical interpolatory geometric subdivision schemes, Computer Aided Geometric Design 80, 101871.}	
		
		\bibitem{Bo}  Borrelli, V., Cazal, F., Morvan, JM.,2003. {\it On the angular defect of
			triangulations and the pointwise approximation of
			curvatures,
			Computer Aided Geometric Design 20, 319-341.}
		
		\bibitem{Bo2}  Borrelli, V., Orgeret, F., 2010. {\it Error term in pointwise approximation of the curvature of a curve, Computer Aided Geometric Design 27, 538-550.}
		
		\bibitem{Ca}  Carlo H. S{\'e}quin and Kiha Lee and Jane Yen, 2005. {\it Fair, $G^2$- and $C^2$-continuous circle splines for the interpolation of sparse data points, Computer-Aided Design 37, 201-211.}
		
		\bibitem{Cas}  Cashman, T. J., Hormann, K. and Reif, U., 2013. { \it Generalized lane–riesenfeld algorithms, Computer Aided Geometric Design 30 (4), 398–409.}
		
		\bibitem{De} Deng, C., Wang, G., 2010. {\it Incenter subdivision scheme for curve interpolation. Computer Aided Geometric Design 27, 48-59.}
		
		\bibitem{We} Deng, C., Ma, W., 2014. {\it  A biarc based subdivision scheme for space curve interpolation. Computer Aided Geometric Design 31, 656-673.}
		
		\bibitem{Dy} Dyn, N., Hormann, K., 2012. {\it  Geometric conditions for tangent continuity of interpolatory
			planar subdivision curves. Computer Aided Geometric Design 29, 332–347.}
		
		\bibitem{Le} Legendre, A. L., 1787. \textit{Sur les opérations trigonométriques dont les résultats dépendent de la figure de la terre. Mém. de l'Acad. de Paris, 352 ff.}
		
		\bibitem{Ma} Marianna Saba, Teseo Schneider, Kai Hormann, Riccardo Scateni,2014.
		{ \it Curvature-based blending of closed planar curves,
			Graphical Models 76, Issue 5,
			Pages 263-272 .}
		
		\bibitem{Su} Masahiro Hirano, Yoshihiro Watanabe, Masatoshi Ishikawa, 2017. { \it Rapid blending of closed curves based on curvature flow,Computer Aided Geometric Design 52–53 217-230.}
		
		\bibitem{Mo} Jean-Marie Morvan, 2008.{ \it  Generalized Curvatures (1st. ed.). Springer Publishing Company, Incorporated.}
		
		\bibitem{Na} Nádeník, Z., 2002. \textit{Legendre theorem on spherical triangles. Research Institute of Geodesy Topography and Cartography.}
		%% \bibitem[Author(year)]{label}
		%% Text of bibliographic item
		%\bibitem[ ()]{}
		
		\bibitem{Sa} Sabin, M. A. and Dodgson, N. A., 2004. { \it A circle-preserving variant of the four-point
			subdivision scheme, Mathematical Methods for Curves and Surfaces: Tromsø pp. 275–
			286.}
		
		
		\bibitem{Vo} Volontè, E., 2017. {\it  Subdivision Schemes
			for Curve Design
			and Image Analysis, Doctoral Dissertation submitted to Università degli Studi Milano Bicocca and Università della Svizzera italiana.}
		
		
	\end{thebibliography}
	
	%% else use the following coding to input the bibitems directly in the
	%% TeX file.

\end{document}